 \documentclass[a4paper,11pt]{article}
\usepackage{pdfsync}

 \usepackage[plainpages=false]{hyperref}
 \usepackage{amsfonts,amsmath,amssymb,amsthm}
 \usepackage{latexsym,lscape,rawfonts,mathrsfs}
 \usepackage{mathtools}
 \usepackage[utf8]{inputenc}
 \usepackage[overload]{empheq}
 \usepackage{cases}

 \usepackage[dvips]{color}
 \usepackage{multicol}

 \usepackage[all]{xy}
 \usepackage{makeidx}
 \usepackage{graphicx,psfrag}
 \usepackage{pstool}
 \usepackage{float}

 \usepackage{array,tabularx}

 \usepackage{setspace}

 \usepackage{appendix}
 
\usepackage{txfonts}

 \newcommand{\R}{\ensuremath{\mathbb{R}}}

 \newcommand{\ba}{\begin{align*}}
 \newcommand{\ea}{\end{align*}}

 \newcommand{\norm}[2]{{ \ensuremath{\left\|} #1 \ensuremath{\right\|}}_{#2}}

 \makeatletter
 \def\ExtendSymbol#1#2#3#4#5{\ext@arrow 0099{\arrowfill@#1#2#3}{#4}{#5}}
 
 \makeatother

 \makeatletter
 \def\ExtendSymbol#1#2#3#4#5{\ext@arrow 0099{\arrowfill@#1#2#3}{#4}{#5}}
 
 \makeatother

 \definecolor{orange}{rgb}{1,0.5,0}
 \definecolor{brown}{rgb}{0.48,0.33,0.19}
 \definecolor{miao}{cmyk}{0.5,0,0.2,0.2}
 \definecolor{qiao}{gray}{0.96}

\newtheorem{prop}{Proposition}[section]

\newtheorem{proposition}[prop]{Proposition}

\newtheorem{theorem}[prop]{Theorem}

\newtheorem{lemma}[prop]{Lemma}

\newtheorem{corollary}[prop]{Corollary}

\newtheorem{remark}[prop]{Remark}

\newtheorem{definition}[prop]{Definition}

\numberwithin{equation}{section}

\title{Ricci curvature integrals, local functionals, and the Ricci flow}
 \author{Yuanqing Ma,  \; Bing Wang}
\date{}

\begin{document}	
	\maketitle
\begin{abstract}

Consider a Riemannian manifold $(M^{m}, g)$ whose volume is the same as the standard sphere $(S^{m}, g_{round})$.
If $p>\frac{m}{2}$ and $\int_{M} \left\{ Rc-(m-1)g\right\}_{-}^{p} dv$ is sufficiently small, we show that the normalized Ricci flow initiated from $(M^{m}, g)$ will exist immortally
and converge to the standard sphere.   The choice of $p$ is optimal. 
\end{abstract}

\tableofcontents


\section{Introduction}

The classical Myer's theorem states that if a closed Riemannian manifold $(M^{m}, g)$ satisfies 
\begin{subequations}
  \begin{align}[left = \empheqlbrace \,]
    & Rc \geq m-1, \label{eqn:OF13_1a} \\
    & |M|_{g} = (m+1) \omega_{m+1},  \label{eqn:OF13_1b}  
  \end{align}
  \label{eqn:OF13_1}
\end{subequations}
\hspace{-2mm}
where $\omega_m$ is the volume of unit ball in the Euclidean space $\R^{m}$. Then the inequality (\ref{eqn:OF13_1a}) must be an equality, and the manifold $(M^{m}, g)$ must be isometric to $(S^{m}, g_{round})$.
It is a natural question to find a quantative version of this rigidity theorem. 
If we replace (\ref{eqn:OF13_1})  by the following conditions
\begin{subequations}
  \begin{align}[left = \empheqlbrace \,]
    & Rc \geq m-1-\delta, \label{eqn:OD14_1a} \\
    & |M|_{g} = (m+1) \omega_{m+1},  \label{eqn:OD14_1b}  
  \end{align}
  \label{eqn:OD14_1}
\end{subequations}
\hspace{-2mm}
for some small number $\delta<\delta_0(m)$, we can still declare many properties of $M$.  
For example, by the deep work of Perelman~\cite{Pe94}, we know that such $M$ must be homeomorphic to $S^{m}$.  
The seminal work of Colding~\cite{Co1}~\cite{Co2} proved that $(M, g)$ is Gromov-Hausdorff close to the standard sphere $(S^{m}, g_{round})$. 
Furthermore, it is proved by the foundational work of Cheeger-Colding~\cite{CC} that $M$ is diffeomorphic to $S^{m}$,  and $(M, g)$ is uniformly bi-H\"older equivalent to $(S^{m}, g_{round})$. 
Their proof relies on the Reifenberg method. 
In~\cite{WB2}, the second named author proved that the normalized Ricci flow initiated from $(M, g)$ exists immortally and converges to a round metric $g_{\infty}$,
and the identity map from $(M, g)$ to $(M, g_{\infty})$ is uniformly bi-H\"older. 
Therefore,  \cite{WB2} provides an alternative proof of Cheeger-Colding' result, via Ricci flow smoothing, which seems to be natural and was intensively studied (e.g. See~\cite{DWY}~\cite{SiTo}~\cite{HuWa} and the references therein.).

The results mentioned above are based on the point-wise Ricci lower bound.
It is interesting to investigate whether this point-wise Ricci lower bound can be replaced by an integral Ricci lower bound.
Actually, many important results based on integral Ricci lower bound were established by the work of Petersen-Wei~\cite{PW1}~\cite{PW2}, D. Yang~\cite{Yan92a}~\cite{Yan92b}~\cite{Yan92c}, S. Gallot~\cite{Gallot}, etc. 
The Ricci flow behavior with initial data satisfying Ricci or other curvature's integral pinching conditions was also studied by many people, e.g., see ~\cite{WYQ09}~\cite{XGY13}~\cite{CCL}. 
Under an appropriate integral Ricci curvature condition(cf. Definition~\ref{dfn:OF16_1} for precise definitions), we have the following theorem.

\begin{theorem}[\textbf{Main Theorem}]
  For each $m \geq 3$ and $p>\frac{m}{2}$, there is a constant $\delta_{0}=\delta_{0}(m, p)$ with the following property.
  
  Suppose $(M^{m}, g)$ is a Riemannian manifold satisfying 
  \begin{subequations}
  \begin{align}[left = \empheqlbrace \,]
    & \int_{M} \left\{ Rc-(m-1)g \right\}_{-}^{p}  dv <\delta, \label{eqn:OH25_13} \\
    & |M|_{g} = (m+1) \omega_{m+1},  \label{eqn:OH25_14}  
  \end{align}
\label{eqn:RG26_1}  
\end{subequations}
\hspace{-3mm}
  where $|\cdot|$ means volume, $\omega_{m+1}$ is the volume of unit ball in $\R^{m+1}$, and $\delta\in (0, \delta_{0})$. 
  Then the normalized Ricci flow initiated from $(M^{m}, g)$ exists immortally and converges to a round metric $g_{\infty}$ exponentially fast.
  The limit metric $(M, g_{\infty})$ has constant sectional curvature $1$. 

  Furthermore, for each pair of points $x,y \in M$ satisfying $d_{g}(x,y) \leq 1$, the following distance bi-H\"older estimate holds:
 \begin{align}
   e^{-\psi} d_{g}^{1+\psi}(x,y) \leq d_{g_{\infty}}(x,y) \leq e^{\psi} d_{g}^{1-\psi}(x,y),    \label{eqn:RG26_2} 
 \end{align}
 where $\psi=\psi(\delta|m, p)$ such that $\displaystyle \lim_{\delta \to 0^{+}} \psi(\delta|m, p)=0$.    
 \label{thm:RG26_3}
\end{theorem}

Note that Theorem~\ref{thm:RG26_3} implies the following topological rigidity result. 

\begin{corollary}[Petersen-Sprouse~\cite{PeSp}, Aubry\cite{Aubry}]
 If $(M^{m}, g)$ satisfies  the condition (\ref{eqn:RG26_1}),  then $M$ is diffeomorphic to standard sphere $S^{m}$ and is close to the standard sphere in the Gromov-Hausdorff topology:
 \begin{align}
      d_{GH} \left\{ (M^{m}, g),  (S^{m}, g_{round})\right\} < \psi(\delta|m,p).    \label{eqn:OH24_1}
 \end{align}
 \label{cly:OH23_1}
\end{corollary}

We remark that the conclusion of Corollary~\ref{cly:OH23_1} was  obtained by Petersen-Sprouse~\cite{PeSp} under the condition (cf. Definition~\ref{dfn:OH25_9})
\begin{align}
 \kappa(1,p,1)  \coloneqq \sup_{x \in M}   \left\{ \frac{\int_{B(x, 1)} \{Rc-(m-1)g\}_{-}^{p} dv}{|B(x, 1)|}\right\}^{\frac{1}{p}} <\xi(m,p)  \label{eqn:OH25_15}
\end{align}
for some small number $\xi=\xi(m,p)$.    The above condition can be deduced from (\ref{eqn:RG26_1}) by the diameter bound of Aubry~\cite{Aubry}.  
One may wonder whether $p$ can be chosen as $\frac{m}{2}$.
However, this expectation is destroyed by the construction of Aubry. 
In view of Theorem 9.2 of~\cite{Aubry},  for arbitrary closed Riemannian manifold $(N^{m}, h)$ satisfying $|N|=(m+1)\omega_{m+1}$ and $\delta>0$, one can construct a Riemannian metric $(M,g)$
of the same volume satisfying
\begin{align}
   d_{GH}\left\{(M,g), (N,h) \right\}<\delta, \quad \textrm{and} \quad \int_{M} \left\{ Rc-(m-1)g \right\}_{-}^{\frac{m}{2}}  dv<\delta. \label{eqn:OI05_1}
\end{align}
Also, it is not hard to see from the construction in~\cite{Aubry} that the topology of $M$ could be far away from $S^{m}$. 
Therefore, the coefficient $p>\frac{m}{2}$ in Theorem~\ref{thm:RG26_3} is optimal. 
It would be a possible generalization to replace $\left\{ Rc-(m-1)g \right\}_{-}$ in equation (\ref{eqn:OI05_1}) by $\left\{ Rc-(m-1)g \right\}$. 
If such generalization holds,  then Corollary~\ref{cly:OH23_1} can be improved to a topological rigidity result in the similar spirit as in Margerin~\cite{Mar}, Bour-Carron~\cite{BoCa2}, and Chang-Gursky-Yang~\cite{CGY}, etc.

We briefly discuss the proof of Theorem~\ref{thm:RG26_3}. 
The first thing need to be done is to obtain uniform existence  time of the Ricci flow solution initiated from $(M, g)$, say $t \geq 1$.
Then we shall show that the volume closeness at $t=0$ implies the Gromov-Hausdorff closeness at $t=0$, and consequently $C^{0}$-closeness
for all $t \in [0.5, 1]$.   By standard regularity improvement, we then obtain $C^{\infty}$-closeness between $(M, g(1))$ and $(S^{m}, g_{round})$.
Then it is well-known (cf.~\cite{Huisken85}~\cite{BoWil}~\cite{BrSch}) that the normalized Ricci flow started from $g(1)$ 
exists forever and converges to the standard sphere $(S^{m}, g_{round})$. 
The strategy described above is the same as the one used in~\cite{WB2}. However, there exist essential technical difficulties to be overcome. 

\textbf{The first difficulty} is to obtain uniform existence time for the Ricci flow initiated from $(M, g)$. 
In~\cite{WB2}, this is achieved through the application of the improved pseudo-locality, whose key is
the estimate of the local $\boldsymbol{\mu}$-functional.   Note that in~\cite{WB2}, the $\boldsymbol{\mu}$-functional can be obtained either through
the blowup analysis, or it can be derived through the isoperimetric constant estimate via the needle-decomposition method for RCD space (cf.~\cite{CaMo}).
However, in the current setting, the Ricci curvature is not bounded from below point-wisely. 
Thus the RCD theory can not be applied directly here.  Fortunately, the blowup analysis method survives, as it only  requires gradient estimates, volume comparison theorems and  compactness properties,  
all of which are available. 

\textbf{The second difficulty} is the distance distortion estimate.  
Note that it is a key  ingredient for the distance distortion estimate to obtain the  the smallness of the space-time integral $\iint |R| dv dt$, which is the almost Einstein condition. 
In ~\cite{WB2}, this condition is obtained through applying maximum principle for the scalar curvature $R$, since the almost non-negativity of $R$ is preserved under the Ricci flow.  In the current paper, we only have an integral estimate of $|Rc_{-}|$ at initial time.
Therefore, the maximum principle argument fails.  We overcome this difficulty by delicately applying the curvature condition
\begin{align}
  t|Rm|(x,t)\leq \epsilon,  
\label{eqn:OF17_1}
\end{align}
where $\epsilon=\epsilon(m)$ is a small number far less than $1$. 
We construct new cutoff functions and develop an ODE system for rough volume ratio upper bound  and $\int |R_{-}|^{p}$. 
Then it follows from ODE comparison that both of them are well-controlled.    A further interpolation argument then implies the almost Einstein condition.

In conclusion,  the first difficulty can be overcome if we are able to apply the pseudo-locality property  in Theorem 1.2 of~\cite{WB2}.
Then it suffices to show a delicate $\boldsymbol{\mu}$-functional estimate under the initial metric, which is achieved by the following theorem. 

\begin{theorem}
Let $(M,g)$ be a closed Riemannian manifold of dimension $m\geq3$ and $p>\frac{m}{2}$. For $x_0\in M$ and each pair of positive numbers $(\eta,A)$, there exists a constant $\delta=\delta(m,p,\eta,A)$ with the following properties. 

Suppose
\begin{subequations}
	\begin{align}[left = \empheqlbrace \,]
		& \int_{B(x_0,\delta^{-1})}Rc_-^p dv\leq \frac{1}{2} \omega_m \delta^{4p-m},   \label{eqn:OH27_2a}\\
		& \left| B\left(x_0, \delta^{-1}\right) \right|\geq(1-\delta) \omega_m\delta^{-m}.  \label{eqn:OH27_2b}
	\end{align}  
\label{eqn:OH27_2}			
\end{subequations}
\hspace{-2mm}
Then we have 
\begin{align}
&\inf\limits_{\tau\in\left[\frac{1}{3}, 3\right] }\boldsymbol{\bar{\mu}}\left( B(x_0,A),g,\tau\right)\geq-\frac{\eta}{2}, \label{eqn:OD12_4}\\
&\inf\limits_{\tau\in\left[\frac{1}{2}, 2\right] }\boldsymbol{\mu}\left( B(x_0,A),g,\tau\right)\geq-\eta. \label{eqn:OF16_2}
\end{align}
\label{thm:OD12_1}
\end{theorem}

Note that (\ref{eqn:OH27_2}) is a version of the ``almost Euclidean" condition. 
(\ref{eqn:OH27_2b}) means the volume ratio is very close to the Euclidean one, (\ref{eqn:OH27_2a}) means that Ricci curvature is almost non-negative in the integral sense.
The coefficient $\frac{1}{2} \omega_m$ in (\ref{eqn:OH27_2a}) is only for technical convenience and could be replaced by other small constant. 
The entropy estimate (\ref{eqn:OF16_2}) assures us to apply the pseudo-locality theorem.   In particular, we shall have the curvature estimate (\ref{eqn:OF17_1}). 
Therefore, under the help of condition (\ref{eqn:OF17_1}), we are able to overcome the second difficulty in light of the following theorem.

\begin{theorem}
  Suppose $\{(M^m,g(t)),  0 \leq t \leq 1\}$ is a Ricci flow solution,  $r_0\geq 3$ and $p>\frac{m}{2}$. 
  Suppose
		\begin{subequations}
			\begin{align}[left = \empheqlbrace \,]
				& \left| B_{g(0)}(x_0,r_0) \right|_{dv_{g(0)}} \leq V_0,   \label{eqn:OH27_1a} \\
				&\int_{B_{g(0)}(x_0,r_0)} R_{-,g(0)}^p dv_{g(0)}\leq \epsilon,   \label{eqn:OH27_1b}
			\end{align}  
		\label{eqn:OH27_1}	
		\end{subequations}
		and 
		\begin{align}
			R_{-, g(t)}(x) \leq \frac{\epsilon}{t}, \quad  \forall \; x \in B_{g(0)}(x_0,r_0), \; t \in (0, 1].       \label{eqn:OF21_5}
		\end{align}
		Then we have
		\begin{subequations}
			\begin{align}[left = \empheqlbrace \,]
				& \left|  B_{g(0)}(x_0,r_0-2)\right|_{dv_{g(t)}} \leq V_0+\frac{(V_0+\epsilon)\left( e^{2Ct}-1\right) }{2} ,   \label{eqn:OF21_3} \\
				&\int_{B_{g(0)}(x_0,r_0-2)} R_{-,g(t)}^p dv_{g(t)} \leq \epsilon+\frac{(V_0+\epsilon)\left( e^{2Ct}-1\right) }{2},   \label{eqn:OF21_2}
			\end{align}  
		\end{subequations}
		for all $t\in [0,1]$. Furthermore,
		\begin{align}
		    \int_{0}^{1}dt\int_{B_{g(0)}(x_0,r_0-2)} R_{-,g(t)}dv_{g(t)}\leq \psi(\epsilon |m,p,V_0),      \label{eqn:OF21_4}
		\end{align}		
      	where $\psi=\psi(\epsilon|m, p)$ such that $\displaystyle \lim_{\epsilon \to 0} \psi(\epsilon|m, p)=0$.
      	\label{thm:OF21_1} 		
\end{theorem}

By Theorem~\ref{thm:OF21_1}, we actually obtain the volume ratio upper bound and  the $L^{p}$-integral bound of scalar curvature along the flow,  under appropriate initial conditions. 
If we have the almost Euclidean condition (\ref{eqn:OH27_2}) at initial time, then conditions  (\ref{eqn:OH27_1}) and (\ref{eqn:OF21_5}) in Theorem~\ref{thm:OF21_1} are all satisfied automatically. 
Furthermore,  the inequality (\ref{eqn:OF21_4}) implies the local almost Einstein condition, via another volume comparison argument. 

     \begin{theorem}
      Suppose $\{(M^m,g(t)),  0 \leq t \leq 1\}$ is a Ricci flow solution,  $m \geq 3$ and $p>\frac{m}{2}$. 
      For $x_0\in M$ and each pair of positive numbers $(\xi,A)$, there exists a constant $\delta=\delta(m,p,\xi,A)$ with the following property. 
     
     If (\ref{eqn:OH27_2}) is satisfied with respect to the initial metric $g(0)$, then we have  
		\begin{align}
		\int_{0}^{1} \int_{B_{g(0)}(x_0, A)} |R| dv_{g(t)} dt \leq  \xi.  
		\label{eqn:OH25_3}
		\end{align}
        \label{thm:OH25_4}		
       \end{theorem}

In light of the local almost Einstein condition (\ref{eqn:OH25_3}), we are able to develop the distance distortion estimate.
Thus the second difficulty is overcome.  Consequently, we can apply the distance distortion estimate and $C^{0}$-continuous dependence of initial data
for the Ricci-Deturck flow, as done in~\cite{WB2}, to obtain the continuous dependence of the metric with respect to the Gromov-Hausdorff topology.

\begin{theorem}[\textbf{Continuous dependence in the $C^\infty$-Cheeger-Gromov topology}]
		Suppose that $\mathcal{N}=\left\lbrace \left( N^m, h(t)\right), 0\leq t\leq T \right\rbrace$ is a (normalized) Ricci flow solution on a closed manifold $N$, initiated from $h(0)=h$, and $p>\frac{m}{2}$.
		For each $\epsilon$ small, there exists $\eta=\eta(\mathcal{N}, p, \epsilon)$ with the following properties.
		
		Suppose $\left( M^m, g\right)$ is a Riemannian manifold satisfying 
		\begin{align}
			\kappa(p,\epsilon)\leq \epsilon^{2-\frac{m}{p}},\quad d_{GH}\left\lbrace (M,g),(N,h)\right\rbrace<\eta  
			\label{eqn:OF21_7}
		\end{align}
		where $p>\frac{m}{2}$. Then the (normalized) Ricci flow initiated from $(M,g)$ exists on $\left[0,T \right]$. Furthermore, there exists a family of diffeomorphisms $\left\lbrace \Phi_t:N\longrightarrow M, \epsilon\leq t\leq T\right\rbrace $ such that 
		\begin{align}
			\sup\limits_{t\in\left[ \epsilon,T\right] }\left\| \Phi_t^*g(t)-h(t)\right\|_{C^{\left[ \epsilon^{-1}\right]}(N,h(t))}<\epsilon. 
		\label{eqn:OF21_8}	
		\end{align}
		
	\label{thm:OF21_6}	
	\end{theorem}

	In Theorem 1.3 of~\cite{WB2}, the author proved the continuous dependence when $Rc_M>-(m-1)A$.
	In the current paper, this condition is replaced by the integral condition $\kappa(p,\epsilon) <\epsilon^{2-\frac{m}{p}}$. 
	Note that if $\epsilon <\left( \frac{1}{(m-1)A}\right) ^{\frac{p}{m}}$ , then $Rc>-(m-1)A$ implies that $\kappa(p,\epsilon) <\epsilon^{2-\frac{m}{p}}$.
	Therefore,  Theorem~\ref{eqn:OF21_8} is an improvement of Theorem 1.3 in~\cite{WB2}. 
	
	It is not hard to see that Theorem~\ref{thm:RG26_3} follows from Theorem~\ref{thm:OF21_6}. 
	In fact, if condition (\ref{eqn:RG26_1}) is satisfied for sufficiently small $\delta$, then  it follows from Petersen-Sprouse~\cite{PeSp} and Aubry~\cite{Aubry} that (\ref{eqn:OF21_7}) holds with $(N^{m},h)=(S^{m}, g_{round})$. 
	Thus  we can apply Theorem~\ref{thm:OF21_6} to obtain that the normalized Ricci flow started from $(M,g)$ exists on $[0,1]$. Furthermore, the metric $g(1)$ is very close to $g_{round}$ in $C^{5}$-topology.
	In particular, the curvature operator of $g(1)$ is nearby the standard one.  Thus we can apply Huisken's theorem~\cite{Huisken85}  to obtain the convergence of the normalized Ricci flow. 
	The remaining argument is the same as that in~\cite{WB2}. 
	
	Reviewing the proof of Theorem~\ref{thm:RG26_3},  it is clear that whenever the major two difficulties (almost Euclidean entropy estimate and almost Einstein condition) are overcome,
	most of the theorems in~\cite{WB2} can be naturally extended to the current situation.  For example, when the metric is Gromov-Hausdorff close to a stable Einstein manifold, 
	then the normalized Ricci flow has global existence and convergence.   The details will be provided at the end of this paper in Section~\ref{sec:stability}.
	We also remark that Theorem~\ref{thm:OD12_1}, Theorem~\ref{thm:OF21_1} and Theorem~\ref{thm:OH25_4} have their independent interests, rather than only being intermediate steps to obtain Theorem~\ref{thm:RG26_3}. \\

	The structure of the paper is described as follows.  In Section~\ref{sec:pre}, we review some elementary materials needed in the paper. 
	In Section~\ref{sec:functionals},  we settle the first difficulty by developing the local $\boldsymbol{\mu}$-functional estimate in Theorem~\ref{thm:OD12_1}. 
	In Section~\ref{sec:SLP}, we overcome the second difficulty and obtain the almost Einstein condition, by proving Theorem~\ref{thm:OF21_1} and Theorem~\ref{thm:OH25_4}. 
	In Section~\ref{sec:distortion}, we follow the strategy of~\cite{WB2} to obtain the distance distortion estimate and we improve the rough distance distortion estimate to $C^{\infty}$-closeness, i.e., Theorem~\ref{thm:OF21_6}. 
	Finally, in Section~\ref{sec:stability}, we study the stability of the Ricci flow and prove Theorem~\ref{thm:RG26_3}.\\

{\bf Acknowledgements}:  

 Both authors are grateful to professor Youde Wang and Jie Wang for helpful discussions. 
 Yuanqing Ma is supported by  YSBR-001, NSFC 11731001 and NSFC 11971400,  Bing Wang is supported by  YSBR-001,  NSFC 11971452, NSFC 12026251 and a research fund of USTC.

\section{Preliminaries}
\label{sec:pre}

In this section, we fix notations and overview important results under the integral Ricci curvature condition.

Let $(M^{m},g)$ be a closed Riemannian manifold of dimension $m$. Let $p$ be a positive number for studying $L^{p}$-norm of Ricci curvature. 
Without further explanation, we always assume by default that
\begin{align}
   m \geq 3, \quad \textrm{and} \quad p>\frac{m}{2}.    \label{eqn:OH25_12}
\end{align}
We shall use $|\cdot|_{dv_{g}}$ to denote volume of domains in the Riemannian manifold $(M,g)$, and the subscription $g$ will be omitted if $g$ is clear in the context.
When no confusion is possible, we also use $|\cdot|$ to denote the area of hypersurfaces  in a Riemannian manifold. 
In the situation we want to highlight the underlying metric to define the volume, we also use $vol_{g} (\cdot)$ to denote volume. 
We use $\omega_m$ to denote the volume of unit ball in $\R^{m}$. Therefore, the standard round sphere of sectional curvature $1$ has volume $(m+1) \omega_{m+1}$. 

The following notations are commonly used in the study of comparison geometry in terms of the integral Ricci curvature condition.

\begin{definition}
	For any $x\in M$, let $\zeta(x)$ be the smallest eigenvalue of the Ricci tensor $Rc:T_x M\rightarrow T_x M$. 
	Then we define
	\begin{align}
		&Rc_{-}(x) \coloneqq \left|\min\left\lbrace0,\zeta(x)\right\rbrace \right|,   \label{eqn:OF29_0}\\
		&\{Rc-(m-1)g\}_{-}(x) \coloneqq \left|\min\left\lbrace0,\zeta(x)-(m-1)\right\rbrace \right|,   \label{eqn:OH25_7}\\
		&R_{-}(x) \coloneqq \left| \min\left\lbrace{0,R}\right\rbrace\right|,  \label{eqn:OH25_8}
	\end{align}	
where $R=tr_{g} Rc$ is the scalar curvature. 	
\label{dfn:OF16_1}	
\end{definition}	

By definition, it is clear that
\begin{align}
   0\leq R_-\leq mRc_{-}.    \label{eqn:OH25_10}
\end{align}

\begin{definition}	
	For $p,r>0$, we define
	\begin{align}
		& \kappa_g(x,p,r):=r^2\left( \fint_{B_g(x,r)} Rc_{-,g}^p dv_g\right)^{\frac{1}{p}}
		=r^2\left( \frac{\int_{B_g(x,r)} Rc_{-,g}^p dv_g}
		{vol_g\left(B_g(x,r)\right)}\right)^{\frac{1}{p}}, \label{eqn:OF29_1} \\
		& \kappa_g(p,r):=\sup\limits_{x\in M}\kappa_g(x,p,r).    \label{eqn:OF29_2}
	\end{align}		
	For a constant $\lambda$,  we can similarly define
	\begin{align}
		& \kappa_g(x,\lambda,p,r):=r^2\left( \fint_{B_g(x,r)} \left|\min\left\lbrace0,\zeta(x)-(m-1)\lambda\right\rbrace \right|^p dv_g\right)^{\frac{1}{p}},  \label{eqn:OF29_1.1} \\
		& \kappa_g(\lambda,p,r):=\sup\limits_{x\in M}\kappa_g(x,\lambda,p,r).    \label{eqn:OF29_2.2}
	\end{align}
	The subscription $g$ will be omitted if there is no ambiguity. 
\label{dfn:OH25_9}	
\end{definition}	

It is clear that $\kappa_g(\lambda,p,r)=\kappa_g(p,r)$ when $\lambda=0$.
Note that $\kappa_{g}(p, r)$ is scaling invariant. Namely, $\kappa_g(p,r)=\kappa_{a^2g}(p,ar)$ for each $a>0$.

\begin{definition}	
	For a function $f$ on $(M, g)$, we define
	\begin{align}
		& \left| \left| f\right| \right|^*_{p,B(x,r)}:= \left( \fint_{B(x,r)}\left| f\right|^p dv\right)^{\frac{1}{p}}=\left( \frac{\int_{B(x,r)}\left| f\right|^p dv}{vol\left(B(x,r)\right)}\right)^{\frac{1}{p}}, \label{eqn:OF29_3} \\
		& \left| \left| f\right| \right|_{p,B(x,r)}:= \left( \int_{B(x,r)}\left| f\right|^p dv \right)^{\frac{1}{p}}.    \label{eqn:OF29_4}
	\end{align}
	Here $\fint$ denotes the mean value of integration.

	\label{def:OF16_1}	
\end{definition}

The Sobolev constant estimate and the non-collapsing condition are closely related. 
The existence of uniform Sobolev constant naturally implies the non-collapsing condition(e.g. Proposition 2.4 of~\cite{YDWang}). 
On the other hand, the non-collapsing condition also implies the estimate of isoperimetric constant and hence the Sobolev constant(e.g. Theorem 3 of~\cite{Gallot} and Theorem 7.4 of~\cite{Yan92a}).  
If we consider the normalized volume, then non-collapsing always holds true formally and so does the Sobolev constant. 
The following estimate was part of Corollary 4.6 of~\cite{DWZ}. 

	\begin{lemma}
	\label{lma:OH20_1}
		There exists a small constant $\epsilon=\epsilon(m,p)>0$ and some $C_S=C(m)r$ such that if $\kappa(p,1)\leq \epsilon$, there holds that
		\begin{equation}\label{2.2}
			\left| \left| f\right| \right|^*_{\frac{2m}{m-2},B(x,r)}\leq C_S\left| \left| \nabla f\right| \right|^*_{2,B(x,r)}
		\end{equation}
		for any $r \in (0, 1)$, $f\in C^{\infty}_0(B(x,r))$ and consequently, for all $f\in W^{1,2}_0(B(x,r))$. 
	\end{lemma}

	Given $z \in M$, let $r(y)=d(y,z)$ be the distance function and 
	\begin{align*}
	   \iota(y) \coloneqq \left( \Delta r-\frac{m-1}{r}\right) _{+}. 
	\end{align*}
	The classical Laplacian comparison states that, if the Ricci curvature $Rc_M$ of $M$ satisfies $Rc_M\geq0$, then $\Delta r\leq \frac{m-1}{r}$, i.e. $\iota\equiv0$. In \cite{PW1}, this result is generalized to the case of integral Ricci bounds. 
	
	\begin{lemma}[{\cite{PW1}}, Lemma 2.2]
		For any $p>\frac{m}{2}$ and $r>0$, there hold
		\begin{align}
			&\left| \left| \iota\right| \right|_{2p,B(z,r)}\leq\left(\frac{(m-1)(2p-1)}{2p-m} \left| \left|Rc_- \right| \right|_{p,B(z,r)} \right)^{\frac{1}{2}},  \label{l2.3}\\
			&\left| \left| \iota\right| \right|^*_{2p,B(z,r)}\leq\left(\frac{(m-1)(2p-1)}{2p-m} \left| \left|Rc_- \right| \right|^*_{p,B(z,r)} \right)^{\frac{1}{2}}. \label{2.4}
		\end{align}	
	\end{lemma}
	
	We also need the following important consequences about volume comparison.
	\begin{lemma}[\textbf{Volume doubling property}, {\cite{PW2}}, Theorem 2.1] 
	\label{l2.4}
		Let $(M,g)$ be a Riemannian manifold of dimension $m\geq3$ and $p>\frac{m}{2}$. For a point $x_0\in M$, if $\kappa(x_0,p, r_2)\leq \epsilon$, there exists a constant $C(m,p)$ which satisfies 
		\begin{equation}\label{2.5}
		   \left\{1-C(m,p)\epsilon^{\frac{1}{2}} \right\} \cdot \frac{ \left|B(x_0,r_2)\right| }{r_2^m}\leq\frac{\left| B(x_0,r_1)\right| }{r_1^m}
		\end{equation}
	    for any $r_1\in (0,r_2]$.
	\end{lemma}
	
	\begin{lemma}
	\label{l2.5}
		Let $(M,g)$ be a closed Riemannian manifold of dimension $m\geq3$ and $p>\frac{m}{2}$. For $x_0\in M$ and a constant $r \geq 1$, there exists a small constant $\delta=\delta(r,m,p)$ and $C=C(m,p)$ with the following properties.
		Suppose
		\begin{subequations}
			\begin{align}[left = \empheqlbrace \,]
				& \int_{B(x_0,\delta^{-1})}Rc_-^p dv\leq \frac{1}{2} \omega_m \delta^{4p-m},  \label{equ:2.1}\\
				& \left| B\left(x_0, \delta^{-1}\right) \right|\geq (1-\delta) \omega_m\delta^{-m}. \label{equ:2.2}
			\end{align}  
		\end{subequations}
		Then for any $x\in B(x_0,r)$ and $s\in (0,\delta^{-1}-r]$, we have
		\begin{subequations}
			\begin{align}[left = \empheqlbrace \,]
				& \kappa(x,p,s)\leq \delta^2,  \label{eqn:OI05_6a}\\
				& \left|B\left( x,s\right) \right|\geq \left(1-C r\delta \right) \omega_ms^{m}. \label{eqn:OI05_6b}
			\end{align}  
		\label{eqn:OI05_6}	
		\end{subequations}	
	\end{lemma}
	
	\begin{proof}
		It follows from (\ref{equ:2.1}) and (\ref{equ:2.2}) that
		\begin{align*}
			\kappa(x_0,p, \delta^{-1})=\delta^{-2}\left( \frac{\int_{B(x,\delta^{-1})} Rc_{-}^p dv}{\left|B(x,\delta^{-1})\right|}\right)^{\frac{1}{p}} < \delta^2.
		\end{align*}
		Thus we can apply Lemma \ref{l2.4} to obtain
		\begin{align}
			\left| B(x_0,s)\right| 
			\geq \left(1-C\delta\right)\left| B\left( x_0,\delta^{-1}\right)\right|  s^m\delta^m 
			\geq (1-C\delta)\omega_ms^m
		\label{eqn:OI05_2}	
		\end{align}	
		for any  $s\in (0,\delta^{-1}]$. 	Fix $x \in B(x_0, r)$. It is clear that 
		\begin{align}
			B(x_0,\delta^{-1}-2r)\subset B(x,\delta^{-1}-r)\subset B(x_0,\delta^{-1}). 
		\label{eqn:OI05_3}	
		\end{align}
		Combining (\ref{eqn:OI05_2}) and (\ref{eqn:OI05_3}), we have
		\begin{align}
			\left|B(x,\delta^{-1}-r)\right| &\geq  \left| B(x_0,\delta^{-1}-2r) \right| \geq
			(1-C\delta)\omega_m(\delta^{-1}-2r)^m
			\geq (1-Cr\delta)\omega_m(\delta^{-1}-r)^m, 
		\label{eqn:OI05_4}	
		\end{align}
		which implies that
		\begin{align}
			&\quad \kappa(x,p, \delta^{-1}-r) \notag\\
			&=(\delta^{-1}-r)^2 \left( \frac{\int_{B(x,\delta^{-1}-r)} Rc_{-}^p dv}{\left|B(x,\delta^{-1}-r)\right|}\right)^{\frac{1}{p}}
			\leq (\delta^{-1}-r)^2\left(  \frac{\int_{B(x,\delta^{-1})} Rc_{-}^p dv}{(1-Cr\delta)\omega_m(\delta^{-1}-r)^m}\right)^{\frac{1}{p}} \notag\\
			&\leq (\delta^{-1}-r)^2\left(  \frac{ \frac{1}{2} \omega_m \delta^{4p-m}}{(1-Cr\delta)\omega_m(\delta^{-1}-r)^m}\right)^{\frac{1}{p}}
			  < (1-\delta r)^{2-\frac{m}{p}} \cdot \delta^2 < \delta^2. 
			\label{eqn:OI05_5}  
		\end{align}
		In light of (\ref{eqn:OI05_4}) and (\ref{eqn:OI05_5}), we can apply the volume comparison in Lemma~\ref{l2.4} again to obtain
		\begin{align}
			s^{-m}\left| B(x,s)\right| 
			\geq \left(1-C\delta\right) (\delta^{-1}-r)^{-m} \left| B\left( x,\delta^{-1}-r\right)\right| 
			\geq (1-Cr\delta)\omega_m
		\label{eqn:OI05_7}	
		\end{align}	
		for each $s\in (0,\delta^{-1}-r]$. 
		Similar to (\ref{eqn:OI05_5}), we deduce that
		\begin{align}
			\kappa(x,p,s)\leq s^2\left( \frac{ \frac{1}{2} \omega_m \delta^{4p-m}}{(1-Cr\delta)\omega_ms^m}\right)^{\frac{1}{p}}
			\leq (s\delta)^{2-\frac{m}{p}} \delta^{2}< \delta^{2}.
		\label{eqn:OI05_8}	
		\end{align}
		It is clear that (\ref{eqn:OI05_6}) is equivalent to (\ref{eqn:OI05_7}) and (\ref{eqn:OI05_8}). The proof of Lemma~\ref{l2.5} is complete. 
	\end{proof}
	
	We also need to estimate the area of the sphere under the integral Ricci curvature condition.
	\begin{lemma}[\textbf{Comparison of volume element}, {cf. \cite{DW}}, Lemma 3.2]
		Given $r>0$ and $p>\frac{m}{2}$, for any $\rho\leq r$, there exists a constant $C=C(m,p,r)$ such that
		\begin{align}
		  &\frac{\left| \partial B(x,\rho)\right|}{\omega_m \rho^{m-1}}
			\leq 1+C(m,p,r)\sup\limits_{x\in M}\left( \int_{B(x,r) }Rc_{-} ^p dv \right)^{\frac{1}{2p}} 
			\leq 1+C(m,p,r)\kappa(p,r)^{\frac{1}{2}}. \label{2.12}\\
		  &\left| \partial B(x,\rho)\right|\leq \left( 1+C(m,p,r)\kappa(p,r)^{\frac{1}{2}}\right) \omega_m \rho^{m-1}.\label{2.13} 
		\end{align}
	\label{l2.6}	
	\end{lemma}
	
	The local functionals of Perelman are defined as follows(cf.~\cite{WB2}). 

       \begin{definition}
		For $\varphi\in W^{1,2}_0(\Omega),\varphi\geq0$ and $\int_\Omega\varphi^{2} dv=1$, as usual, we define the local entropy functionals
		\begin{align}
		 &\boldsymbol{\mathcal{W}}(\Omega,g,\varphi,\tau):=-m-m\log\sqrt{4\pi\tau}+\int_{\Omega} \left\{ \tau\left(R \varphi^2+4\left| \nabla\varphi\right|^2 \right)-2\varphi^2\log\varphi \right\} dv, \label{eqn:OF28_9}\\
		 &\boldsymbol{\overline{\mathcal{W}}}(\Omega,g,\varphi,\tau):=-m-m\log\sqrt{4\pi\tau}+\int_{\Omega} \left\{ 4\tau\left| \nabla\varphi\right|^2-2\varphi^2\log\varphi \right\} dv, \label{eqn:OF28_10}\\
		 &\boldsymbol{\mu}(\Omega,g,\tau):=\inf\limits_{\varphi}\boldsymbol{\mathcal{W}}(\Omega,g,\varphi,\tau), \label{eqn:OF28_11}\\
		 &\boldsymbol{\bar{\mu}}(\Omega,g,\tau):=\inf\limits_{\varphi}\boldsymbol{\overline{\mathcal{W}}}(\Omega,g,\varphi,\tau). \label{eqn:OF28_12}
		\end{align}
	 \label{d3.1}		
	\end{definition}

        As we already know(cf.~\cite{WB2}), the estimates of local functionals are crucial for taming the behavior of Ricci flows. 
	The following gradient estimate, obtained recently by Jie Wang and Youde Wang, is important for estimating the local functionals. 
			
	\begin{theorem}[J.Wang-Y.D.Wang~\cite{WJWYD}]
	\label{t3.2}
		Suppose $m\geq3$, $p>\frac{m}{2}$; or $m=2$, $p>\frac{3}{2}$. 
		There exists a small positive constant  $\epsilon=\epsilon(m,p)>0$ with the following properties.
		
		Suppose $(M^{m}, g)$ is a complete Riemannian manifold satisfying  $\kappa(p,1)\leq \epsilon$.
		Suppose $u$ is a positive solution of 
		\begin{equation}
		\label{3.1}
		\Delta u+\frac{1}{2}u\log u+Au=0
	       \end{equation}
	       satisfying $0<u\leq D$.  Suppose $B(x,1) \cap \partial M=\emptyset$.
	       Then for each  $0<\lambda\leq1$ and $q>1$, we have
		\begin{equation}
		\label{3.2}
			\sup\limits_{B\left( x,\frac{\lambda}{2}\right) } \frac{\left| \nabla u\right|}{u^{1-\frac{1}{2q}}}
			\leq \frac{C}{\lambda}
		\end{equation} 
	       where $C=C\left( m,p,q,A,D,\kappa(p,1)\right)$. 
	\end{theorem}
	
	Note that equation (\ref{3.1}) is the Euler-Lagrange equation of the local $\bar{\boldsymbol{\mu}}$-functional.  The gradient estimate (\ref{3.2})  does not need the non-collapsing condition and has independent interest.\\
	
	The following weak Harnack inequality,  introduced by P. Li in \cite{Li}, will be very useful in this paper. 
	
		\begin{lemma}[{\cite{Li}, Lemma 11.2}]
		\label{l3.3}
		Let $(M,g)$ be a complete Riemannian manifold of dimension $m\geq3$. Suppose that the geodesic ball $B(x,r)$
		satisfies $B(x,r)\cap\partial M =\phi$. Let $u\geq0$ be a function in the Sobolev space $W^{1,2}\left( B_r\right)$ and satisfy
		\begin{equation}\label{3.3}
			\Delta u\leq Au
		\end{equation} 
		in the weak sense for some constant $A\geq0$ on $B_r$.  Suppose the following conditions hold true:
		
		(i) Normalized Sobolev inequality:
		\begin{equation}\label{3.4}
			\left| \left| f\right| \right|^*_{\frac{2m}{m-2},B(x,r)}\leq C_S\left| \left| \nabla f\right| \right|^*_{2,B(x,r)}
		\end{equation}
		for any compactly supported $f\in W^{1,2}\left( B_r\right)$ where $C_S=C_s(m)r>0$.
		
		(ii) Poincar\'{e} inequalities:
		\begin{equation}\label{3.5}
			4C_P\int_{B\left( x,\frac{r}{2}\right) }\left| f-f_{B_\frac{r}{2}}\right|^2 dv \leq r^2\int_{B\left( x,\frac{r}{2}\right) }\left| \nabla f\right|^2 dv 
		\end{equation}
		for any $f\in W^{1,2}\left( B_\frac{r}{2}\right)$ and
		\begin{equation}\label{3.6}
			16C_P\int_{B\left( x,\frac{r}{4}\right) }\left| f-f_{B_\frac{r}{4}}\right|^2 dv \leq r^2\int_{B(x,\frac{r}{4})}\left| \nabla f\right|^2 dv 
		\end{equation}
		for any $f\in W^{1,2}\left( B_\frac{r}{4}\right)$ where $f_{B_r}=\fint_{B_r}f$ and $C_P>0$ is a uniform constant which does not depend on $r$.
		
		(iii)  Volume doubling property:
		\begin{equation}\label{3.7}
			\frac{\left| B_r\right|}{\left| B_{\frac{r}{16}}\right|}\leq C_V
		\end{equation}
		where $C_V$ is a uniform constant which does not depend on $r$.
		
		Then for $\theta>0$ sufficiently small, there exists a constant $C$ depending only on $\theta, m, C_S,C_P,C_V$ and $Ar^2+1$ such that 
		\begin{equation}\label{3.8}
			\left| \left| u\right| \right|^*_{\theta,B_{\frac{r}{8}}}\leq C\inf\limits_{B_{\frac{r}{16}}}u.
		\end{equation}
	\end{lemma}
	
	\begin{remark}
		In fact, after some technical improvements, (\ref{3.8}) holds true for any $0<\theta<\frac{m}{m-2}$, cf. Theorem 8.18 of \cite{GT} and step 2 in Theorem 4.15 of \cite{HL}.
	\end{remark}
	
	Under the Ricci curvature condition $\kappa(p,D)<\epsilon(m,p)$ for some $\epsilon$, 
	we obtain (\ref{3.4}) and (\ref{3.7}) by Lemma~\ref{lma:OH20_1} and Lemma~\ref{l2.4} respectively. To obtain the Poincar\'{e} inequalities, we need the following result proved by A. Grigor'yan in \cite{Gr} and L. Saloff-Coste in \cite{Sal1}.
	
	\begin{theorem}[{\cite{Sal2}, Theorem 5.5.1}]\label{t3.5}
		Fix $0<D\leq\infty$ and consider the following properties:
		
		(i) Poincar\'{e} inequality: 
		
		There exists a uniform constant $P_0$ such that, for any ball $ B(x, r)$, $0<r<D$ and all $f\in C^\infty\left(B_r\right)$,
		\begin{equation}\label{3.9}
			\int_{B(x,r)}\left| f-f_{B_r}\right|^2 dv\leq P_0r^2\int_{B(x,r)}\left| \nabla f\right|^2 dv.  
		\end{equation}
		
		(ii)  Volume doubling property:
		
		There exists a uniform constant $D_0$ such that, for any ball $B(x, r)$, $0<r<D$,
		\begin{equation}\label{3.10}
			\left|B_{2r} \right|\leq D_0\left| B_r\right|.  
		\end{equation}
		
		(iii) Parabolic Harnack inequality:
		
		There exists a uniform constant $H_0$ such that, for any ball $B(x, r)$, $0<r<D$, and for any smooth positive solution u to $\Delta u-\partial_t u=0$ in
		$Q=\left( s-r^2, s\right) \times B(x, r)$ with arbitrary real number $s$,
		such that
		\begin{equation}\label{3.11}
			\sup\limits_{Q^-}u\leq H_0\inf\limits_{Q^+}u
		\end{equation}
		where $Q^-=\left( s-\frac{3}{4}r^2, s-\frac{1}{2}r^2\right) \times B\left( x, \frac{r}{2}\right) $ and $Q^+=\left( s-\frac{1}{4}r^2, s\right) \times B\left( x, \frac{r}{2}\right) $.
		
		Then (iii) is equivalent to (i) and (ii).
	\end{theorem}
	
	As Theorem 5.5 of \cite{DWZ} asserts, by the parabolic gradient estimate obtained in \cite{ZM} and a scaling argument, (iii) of Theorem \ref{t3.5} holds true provided $\kappa(p,D)<\epsilon(n,p)$ for some small $\epsilon$ and $H_0=H_0(m,p,D)$, then we obtain the desired Poincar\'{e} inequalities in Lemma \ref{l3.3} under Ricci curvature bounds at once.
	
	As a consequence, Lemma \ref{l3.3} holds true and we have the following estimate analogous to Proposition 2.2 of \cite{TW}.
	\begin{proposition}\label{p3.6}
		Let $(M,g)$ be a complete Riemannian manifold and $B(x,2)\subset M$. Suppose $\varphi\geq0$ satisfies $(-\Delta+a)\varphi\geq b$ with $\left| a\right|+ \left|b \right|<C_F$ for some constant $C_F$ on $B(x,2).$ Then for any $0<\rho\leq 1$ and $\theta$ small enough, there exists a constant $\epsilon(n,p)$ such that if $\kappa(p,8)<\epsilon$, there holds
		\begin{equation}\label{3.12}
			\rho^{-m}\int_{B_{\rho}}\varphi^\theta dv \leq C^\theta\left( \rho^2+\inf\limits_{B_{\frac{\rho}{2}}}\varphi\right)^\theta
		\end{equation}
		for some constant $C=C(m,p,C_F,C_s(m),C_P,C_V)$ where $\theta$ is the same as in Lemma \ref{l3.3}.
	\end{proposition}
	\begin{proof}
		Let $\bar{\varphi}=\varphi+C_F$, then
		\begin{align*}
				(-\Delta+a)\bar{\varphi}&\geq b+aC_F=C_F\left( a+C_F^{-1}b\right)\geq-C_F\left| a+C_F^{-1}b\right|\geq -\bar{\varphi}\left| a+C_F^{-1}b\right|\\
				&\geq-\bar{\varphi}\left( \left| a\right| +1\right).
		\end{align*}
		It follows that
		$$\Delta\bar{\varphi}\leq\left( 2\left| a\right| +1\right)\bar{\varphi}\leq\left( 2C_F +1\right)\bar{\varphi}.$$
		By Lemma \ref{l3.3}, there holds
		\begin{equation}\label{3.13}
			\fint_{B_1}{\bar{\varphi}}^\theta dv \leq \left(C \inf\limits_{B_\frac{1}{2}}\bar{\varphi}\right)^\theta
		\end{equation}
		for some constant $C$ depending on $m,p,C_F,C_s(m),C_P,C_V$.
		
		Fix $0<\rho\leq1$ and let $\tilde{g}=\rho^{-2}g$. By scaling, we have
		$$(-\Delta+a)\varphi\geq b\Leftrightarrow-\rho^{-2}\Delta_{\tilde{g}}\varphi+a\varphi\geq0\Leftrightarrow-\Delta_{\tilde{g}}\left( \rho^{-2}\varphi\right)+\rho^2a\left( \rho^{-2}\varphi\right)\geq b.$$
		Let $\tilde{\varphi}=\rho^{-2}\varphi$.  Then we have
		$$-\Delta_{\tilde{g}}\tilde{\varphi}+\rho^2a\tilde{\varphi}\geq b.$$
		Note that (\ref{3.13}) still holds true on $(M,\tilde{g})$ after scaling. Consequently, we obtain
		$$\fint_{B_{\tilde{g}}(x,1)}\left( {\tilde{\varphi}}+C_F\right)^\theta dv \leq C^\theta\left( \inf\limits_{B_{\tilde{g}}\left( x,\frac{1}{2}\right) }\tilde{\varphi}+C_F\right) ^\theta,$$
		which yields that
		$$\fint_{B_{g}(x, \rho)}{\varphi}^\theta dv \leq C^\theta\left( \inf\limits_{B_{g}\left( x,\frac{\rho}{2}\right) }\varphi+\rho^2\right) ^\theta$$ for some constant $C$.
		
		By (\ref{2.5}) for $1-C(m,p)\epsilon\geq \frac{1}{2}$, we deduce that $\left| B(x,\rho)\right| \leq2\omega_m\rho^m$ for some $\epsilon$, then we obtain (\ref{3.12}) at once.\\
	\end{proof}	
	The following generalization of Cheeger-Colding theory was achieved by Petersen-Wei~\cite{PW2}. 
	
	\begin{lemma}[{\cite{PW2}, Theorem 1.3}]
	\label{l3.7}
		Suppose a sequence of complete Riemannian m-manifolds $ (M_i,g_i)$ converges to a Riemannian m-manifold $(M,g)$ in the pointed Gromov-Hausdorff topology. 
		Then we can find an $\epsilon(m,p)>0$ such that if for all the manifolds we have $\kappa(p,D)\leq\epsilon$ and the points $x_i\in M_i$ converge to $x\in M$, then
		\begin{align}
		   \left| B(x_i,r)\right|\longrightarrow \left| B(x,r)\right|, \quad \forall \; 0<r< \frac{D}{8}.  \label{eqn:OH27_3}
		\end{align}
		Consequently, a standard covering argument then implies
		\begin{align}
		   |M_i| \longrightarrow |M|.    \label{eqn:OH27_4}
		\end{align}
	\end{lemma}
	
	\begin{lemma}[{\cite{PW2}, Theorem 1.5}]
	\label{l3.8}
		Given $\epsilon>0$, $D>0$, then we find some $\delta_1(m,p,\epsilon)>0$ and $\delta_2(m,p,\epsilon)>0$ such that if $(M,g)$ is a Riemannian m-manifold with $\kappa(p,D)\leq\delta_1$ 
		and $\left| B(x,D)\right|\geq(1-\delta_2)\omega_mD^m$ for some $x\in M$ where $\omega_m$ is the volume of $B_1$ in m-Euclidean space, then $B(x,r)$, $r<\frac{D}{8}$, 
		is $\epsilon$-Gromov-Hausdorff close to an $r$-ball in the m-Euclidean space.
	\end{lemma}
	
	\begin{lemma}[\cite{Aubry}, Theorem 1.1]
	    Let $(M^{m}, g)$ be a complete manifold and $p>\frac{m}{2}$. Then we have
	    \begin{align}
	         |M| \leq (m+1) \omega_{m+1} \left\{  1+ \rho_{p}^{\frac{9}{10}}\right\} \left\{1+ C(p,m) \rho_{p}^{\frac{1}{10}} \right\},
	    \label{eqn:OH24_2}     
	    \end{align}
	    where we denote
	    \begin{align*}
	        \rho_{p} \coloneqq \int_{M} \left\{ Rc-(m-1)g \right\}_{-}^{p} dv.  
	    \end{align*}
	\label{lma:OH24_3}    
	\end{lemma}


\section{Estimate of local functionals}
\label{sec:functionals}
	
	In this section, we shall develop a delicate lower bound of the local $\boldsymbol{\mu}$ functional.
	Our proof is similar to the proof of Proposition 3.1 of \cite{TW} (see also Lemma 4.10 of \cite{WB2}).
	Since there are extra difficulties caused by the Ricci integral condition, we shall modify and streamline the previous proof,  and provide full details. 
	
	\begin{proposition}
		Let $(M,g)$ be a closed Riemannian manifold of dimension $m\geq3$ and $p>\frac{m}{2}$. For $x_0\in M$ and each pair of positive numbers $(\eta,A)$, there exists a constant $\delta=\delta(m,p,\eta,A)$ with the following properties:
		
		Suppose
		\begin{subequations}
			\begin{align}[left = \empheqlbrace \,]
				&\int_{B(x_0,\delta^{-1})}Rc_-^p dv\leq \frac{1}{2} \omega_m \delta^{4p-m},   \label{eqn:OF29_5a} \\
				&\left| B\left(x_0, \delta^{-1}\right) \right|\geq(1-\delta) \omega_m\delta^{-m}.   \label{eqn:OF29_5b}
			\end{align}  
		\end{subequations}
		Then we have 
		\begin{align}
		\label{3.16}
			\inf\limits_{\tau\in\left[\frac{1}{3},3\right] }\boldsymbol{\bar{\mu}}\left( B(x_0,A),g,\tau\right)\geq-\frac{\eta}{2}.
		\end{align}
		
	\label{prn:OF29_7}	
	\end{proposition}
	
	\begin{proof}
	
	        The proof follows similar strategy as that of Proposition 3.1 of~\cite{TW}.   However, we shall encounter new technical difficulties caused for lack of point-wise Ricci lower bound. 
	        Furthermore, we now have more  systematical notations and estimates of local functionals by the work of~\cite{WB} and~\cite{WB2}.  
	        We shall provide more details and intermediate steps than the proof in~\cite{TW}, to make the proof more streamlined. 
	        In the proof, if it is not mentioned otherwise,  $C$ by default denote a constant depending only on $m,p$ and $A$.
	        As usual,  the actual value of $C$ may change from line to line. 
	
	        By monotonicity of local functionals(e.g. See Proposition 2.1 of~\cite{WB}), we may assume $A \geq 1$. 
		By scaling and redefining $\delta$ if necessary, it suffices to show (\ref{3.16}) for $\tau=1$. Namely, it suffices to prove
		\begin{align}
			\boldsymbol{\bar{\mu}}\left( B(x_0,A),g, 1\right)\geq-\frac{\eta}{2}.     \label{eqn:OH14_1}
		\end{align}
		We shall prove (\ref{eqn:OH14_1}) by a contradiction argument.     For simplicity of notation,  in this proof, we denote
		\begin{align}
		   \Omega \coloneqq B(x_0, A).    \label{eqn:OH14_6}
		\end{align}
		Therefore, if the Proposition was wrong, we should have
		\begin{align}
			\boldsymbol{\bar{\mu}}\left(\Omega, g, 1\right) < -\frac{\eta}{2}    \label{eqn:OH14_3}
		\end{align}
		no matter how small $\delta$ is in (\ref{eqn:OF29_5a}) and (\ref{eqn:OF29_5b}). 
		Let $\varphi$ be a minimizer of  the functional $\boldsymbol{\bar{\mu}}\left( \Omega, g,1\right)$.  It satisfies the normalization condition and the Euler-Lagrange equation on $\Omega$:
		\begin{align}
		 &\int \varphi^2 dv=1,  \label{eqn:OH11_1}\\
		 &-4\Delta\varphi-2\varphi\log\varphi-m\left(1+\log\sqrt{4\pi}\right)\varphi=\boldsymbol{\bar{\mu}} \varphi. \label{3.17}
		\end{align}
		Note that $\varphi$ is continuous on $\bar{\Omega}$ and vanishes on $\partial \Omega$.  Therefore,  $\varphi$ can be regarded as a function on $M$ by trivial extension. 
		The equation (\ref{3.17}) can be rewritten as
		\begin{align}
		   &\boldsymbol{\bar{\lambda}}=\frac{m\left(1+\log\sqrt{4\pi}\right) + \boldsymbol{\bar{\mu}}}{4},  \label{eqn:OH14_9}\\ 
		   &-\Delta \varphi -\frac{1}{2} \varphi \log \varphi =\boldsymbol{\bar{\lambda}} \varphi.   \label{eqn:OH14_10}
		\end{align}

		The proof consists of several steps. \\
		
		\noindent
		\textit{Step 1. There is a constant $C=C(m,A)$ such that
		\begin{align}
		      &\left|\boldsymbol{\bar{\mu}} \right|\leq C,   \label{eqn:OH14_2}\\
		      &\norm{\varphi}{C^0(\Omega)} \leq C.    \label{eqn:OH14_4}
		\end{align}
		}

		In view of Lemma \ref{l2.4}, conditions (\ref{eqn:OF29_5a}) and (\ref{eqn:OF29_5b}) assure the uniform non-collapsing condition.
		Then we can apply Lemma~\ref{lma:OH20_1} to obtain uniform Sobolev constant estimate, which in turn implies that $\boldsymbol{\bar{\mu}}$ is uniformly bounded from below by a constant depending on $m$ and the Sobolev constant of $\Omega$. 
		Since $\boldsymbol{\bar{\mu}}\left(\Omega, g, 1\right) \leq 0$  by (\ref{eqn:OH14_3}), this lower bound of  $\boldsymbol{\bar{\mu}}$ yields (\ref{eqn:OH14_2}).   
		Combining (\ref{eqn:OH14_2}) with (\ref{eqn:OH11_1}) and (\ref{3.17}), we can apply standard Moser iteration argument to obtain (\ref{eqn:OH14_4}). 
		For further details, see the proof of Proposition 3.1 in~\cite{TW}. \\

		\noindent
		\textit{Step 2.  There holds that
		\begin{align}
		  \sup_{d(z, \partial \Omega) \geq 0.1}    |\nabla \varphi|(z) \leq C.   \label{eqn:OH18_1}
		\end{align}
		}
		
		Since $\varphi$ satisfies the equation (\ref{eqn:OH14_10}) in $B(z, 0.1) \subset \Omega=B(x_0, A)$ and $\boldsymbol{\bar{\lambda}}$ is uniformly bounded, we can apply the gradient estimate of Youde Wang and Jie Wang(cf. Theorem~\ref{t3.2}).
		By setting $q=m$, it is clear that (\ref{eqn:OH18_1}) follows from  (\ref{3.2}) directly. \\

		\noindent
		\textit{Step 3.  For each fixed $\zeta \in (0, 0.1)$ and $z \in \Omega$ such that $d(z, \partial \Omega) \leq \zeta$,   we have
		\begin{align}
		     &\varphi(z) \leq C \zeta^{\alpha},   \label{eqn:OH14_5}\\
		     &|\nabla \varphi|(z) \leq C \zeta^{\alpha-1},   \label{eqn:OH14_14}
		\end{align}
		whenever $\delta$ is sufficiently small. Here $\alpha=\alpha(m, p) \in (0, 1)$. 
		}

		Note that (\ref{eqn:OH14_5}) is the boundary H\"older estimate. We shall follow the argument in section 8.10 of~\cite{GT} to achieve the proof. 
		Let $w$ be a point in $\partial \Omega$ such that $d(z, \partial \Omega)=d(z,w)$.  For each $r \in (0, 1)$, we define
		\begin{align}
		 &M_{r} \coloneqq Osc_{B(w,r)}(\varphi)=\sup_{B(w,r)} \varphi - \inf_{B(w,r)} \varphi,  \label{eqn:OH14_7}\\
		 &h_{r}\coloneqq M_{2r}-\varphi.    \label{eqn:OH14_8}
		\end{align}
		Then there holds
		\begin{align}
			\left( -\Delta-\boldsymbol{\bar{\lambda}} \right)h_{r}
			  =-\boldsymbol{\bar{\lambda}} M_{2r} -\frac{\left(M_{2r}-h_{r}\right)\log\left( M_{2r}-h_{r}\right)}{2}
			\geq -C,
		\end{align}
		where we used the fact that both $M_{r}$ and $h_{r}$ are uniformly bounded by (\ref{eqn:OH14_4}).

		Since $\delta$ is sufficiently small,  we can apply Proposition~\ref{p3.6} to obtain
		\begin{equation}\label{3.29}
			\left(2 r\right)^{-m}\int_{B(w,2r)}h_{r}^\theta dv \leq C \left( \inf\limits_{B(w,r)}h_{r}+r^2\right)^\theta 
		\end{equation}
		for some $\theta=\theta(p,m)$ small. 
		By compactness and volume convergence, it is clear that 
		\begin{align*}
		|B(w,2r) \backslash \Omega| =  |B(w,2r) \backslash B(x_0, A)| \geq 10^{-m} \omega_m (2r)^{m}. 
		\end{align*}
		Note that 
		\begin{align*}
		  &h_{r}\geq 0 \; \textrm{on} \; B(w,2r)\cap \Omega;\\
		  &h_{r}=M_{2r} \; \textrm{on} \; B(w,2r)\setminus \Omega;\\
		  &\inf\limits_{B(w,r)}h_{r}=M_{2r}-M_{r}.
		\end{align*}
		Therefore (\ref{3.29}) implies that
		\begin{align*}
		   M_{2r}\leq H \left(M_{2r}-M_{r}+r^2 \right)
		\end{align*}
		for some constant $H=H(m,p,\theta)=H(m,p)$ sufficiently large. 
		The above inequality can be rewritten as 
		\begin{align}
		&\gamma \coloneqq 1-\frac{1}{H} \in \left(\frac{1}{2}, 1 \right); \label{eqn:OH14_11}\\
		&M_{r}\leq\left(1- \frac{1}{H}\right)M_{2r}+r^2.    \label{eqn:OH14_12}
		\end{align}
		Fix $r=2^{-i}$ for $i\geq2$.   It follows from (\ref{eqn:OH14_12}) and induction that 
		\begin{equation}\label{3.31}
		M_{2^{-i}}\leq\gamma M_{2^{-i+1}}+4^{-i}\leq\gamma^{i-1}M_{\frac{1}{2}}+\sum^{i-2}_{j=0}\gamma^j4^{-i+j}=\gamma^{i-1}M_{\frac{1}{2}}+\frac{\gamma^{i-1}-4^{1-i}}{4(4\gamma-1)} \leq C \gamma^{i}, 
		\end{equation}
		where we used the uniform boundedness of $M_{\frac{1}{2}}$ in terms of (\ref{eqn:OH14_4}) and (\ref{eqn:OH14_7}).  Consequently, we have
		\begin{equation}\label{3.32}
			\left\|\varphi \right\|_{L^\infty\left( B(w,2^{-i})\right)}\leq  M_{2^{-i+1}}\leq C\gamma^{i}.
		\end{equation}
		By choosing $i$ such that $\zeta \in [2^{-i}, 2^{-i+1})$ and setting
		\begin{align}
		    \alpha' \coloneqq -\frac{\log \gamma}{\log 2}\in\left( 0,1\right),   \label{eqn:OH18_12}
		\end{align}
		then we have
		\begin{align}
		     \varphi(z) \leq C \zeta^{\alpha'}.    \label{eqn:OH14_13}
		\end{align}
		Following the route to prove  (\ref{eqn:OH18_1}), we are ready to  deduce (\ref{eqn:OH14_14}) from the above inequality. 
		Actually, since $\varphi$ satisfies the equation (\ref{eqn:OH14_10}) in $B(z, \zeta) \subset \Omega=B(x_0, A)$ and $\boldsymbol{\bar{\lambda}}$ is uniformly bounded, we can apply Theorem~\ref{t3.2} again. 
		Setting $q=m$ in (\ref{3.2}),  it follows from (\ref{eqn:OH14_13}) that
		\begin{align}
		   |\nabla \varphi|(z) \leq \frac{C}{\zeta} \varphi^{1-\frac{1}{2m}}(z)<C \zeta^{(1-\frac{1}{2m})\alpha'-1}.    \label{eqn:OH14_15}
		\end{align}
		Now we define
		\begin{align}
		   \alpha \coloneqq  \left(1-\frac{1}{2m} \right) \alpha' \in (0, 1).      \label{eqn:OH18_13}
		\end{align}
		Since $\zeta \in (0, 1)$,  it is clear that (\ref{eqn:OH14_5}) and (\ref{eqn:OH14_14}) follow from (\ref{eqn:OH14_13}) and (\ref{eqn:OH14_15}) respectively. \\

		\noindent
		\textit{Step 4.  Fix $z \in \Omega$, $\zeta \in (0, 0.1)$ such that $B(z, 2\zeta) \subset \Omega$.
		 Let $d \coloneqq d(z, \cdot)$.  Then 
		\begin{align}
		    \int_{\Omega \backslash B(z, \zeta)} \left| \Delta d^{2-m} \right| dv < \psi(\delta |m,p,A). 
		\label{eqn:OH14_16}    
		\end{align}
		}
		
		In the sense of distributions, the singular part of $\Delta d$ is a nonpositive measure supported on the cut locus of $z$(cf. \cite{Ch}, Theorem 4.1). 
		Define
		\begin{align}
		   \iota \coloneqq \left( \Delta d - \frac{m-1}{d}\right)_{+} \geq 0.    \label{eqn:OH14_16a}
		\end{align}		
		Direct computation shows that 
		\begin{align}
			  \Delta d^{2-m}+(m-2)d^{1-m} \iota			
			=(2-m)d^{1-m}\left(\Delta d-\frac{m-1}{d}-\iota\right)\geq0.  \label{3.37}
		\end{align}
		Since $\Omega=B(x_0, A) \subset B(z,2A)$, it follows from (\ref{eqn:OH14_16a}) and (\ref{3.37}) that
		\begin{align*}
			\int_{\Omega\setminus B(z,\zeta)}\left| \Delta d^{2-m}\right| dv
			&=\int_{\Omega\setminus B(z,\zeta)}\left| \Delta d^{2-m}+(m-2)d^{1-m}\iota-(m-2)d^{1-m}\iota\right| dv \nonumber\\
			&\leq\int_{\Omega\setminus B(z,\zeta)}\left| \Delta d^{2-m}+(m-2)d^{1-m}\iota\right| dv+(m-2)\int_{\Omega\setminus B(z,\zeta)}\left| d^{1-m}\iota\right| dv\nonumber\\
			&\leq\int_{B(z,2A)\setminus B(z,\zeta)}\left\{ \Delta d^{2-m}+(m-2)d^{1-m}\iota\right\} dv +(m-2)\int_{B(z,2A)\setminus B(z,\zeta)} \left\{ d^{1-m}\iota \right\} dv\nonumber\\
			&=\int_{B(z,2A)\setminus B(z,\zeta)} \left\{ \Delta d^{2-m}\iota \right\} dv +2(m-2)\int_{B(z,2A)\setminus B(z,\zeta)} \left\{ d^{1-m}\iota \right\} dv. 
		\end{align*}
		On the other hand,  Green's formula implies that
		\begin{equation}
			\int_{B(z,2A)\setminus B(z,\zeta)} \left\{ \Delta d^{2-m} \right\} dv=(m-2)\left\lbrace \left| \partial B(z,2A)\right|(2A)^{1-m}-\left| \partial B(z,\zeta)\right|\zeta^{1-m}\right\rbrace. 
		\end{equation}
		From Lemma \ref{l2.6}, we have 
		\begin{equation}\label{3.41}
			\int_{B(z,2A)\setminus B(z,\zeta)} \left\{ \Delta d^{2-m}  \right\} dv \leq \psi(\delta|m,p,A). 
		\end{equation}
		
		By H\"{o}lder's inequality and Lemma {\ref{l2.3}}, it follows that
		\begin{align}
			& \int_{B(z,2A)\setminus B(z,\zeta)} \left\{ d^{1-m}\iota \right\} dv\nonumber\\
			\leq & \left( \int_{B(z,2A)\setminus B(z,\zeta)}d^{\frac{2p(1-m)}{2p-1}} dv\right)^{\frac{2p-1}{2p}}\left( \int_{B(z,2A)\setminus B(z,\zeta)}\iota^{2p} dv\right)^{\frac{1}{2p}}\nonumber\\
			\leq & \left( \int_{\zeta}^{2A}\left( \rho^{\frac{2p(1-m)}{2p-1}}\left| \partial B(z_k,\rho)\right| \right)d\rho \right)^{\frac{2p-1}{2p}}\left(\frac{(m-1)(2p-1)}{2p-m} \left| \left|Rc_-\right| \right|_{p,B(z,2A)} \right)^{\frac{1}{2}}\nonumber\\
			\leq & C \left| \left|Rc_- \right| \right|_{p,B(z,2A)}^{\frac{1}{2}}\left(\int_{\zeta}^{2A}\left( \rho^{\frac{1-m}{2p-1}}\frac{\left| \partial B(z,\rho)\right|}{\rho^{m-1}} \right)d\rho \right)^{\frac{2p-1}{2p}}. \label{3.42}
		\end{align}
		In view of the curvature condition (\ref{eqn:OF29_5a}), we can apply Lemma {\ref{l2.6}} to obtain that 
		\begin{align}
		\frac{\left| \partial B(z,\rho)\right|}{\rho^{m-1}}\leq C   \label{eqn:OH14_17}
		\end{align}
		for all $\rho\leq2A$.  Thus the combination of (\ref{eqn:OH14_17}) and (\ref{3.42}) implies that 
		\begin{equation}
		\label{3.43}
			\int_{B(z,2A)\setminus B(z,\zeta)} \left\{ d^{1-m}\iota \right\} dv \leq C  \left| \left|Rc_- \right| \right|_{p,B(z,2A)}^{\frac{1}{2}}\left(\int_{\zeta}^{2A}\left( \rho^{\frac{1-m}{2p-1}}\right)d\rho \right)^{\frac{2p-1}{2p}}.
		\end{equation}
		Since $-1<\frac{1-m}{2p-1}<0$ as $p>\frac{m}{2}$,  it is clear that $\left(\int_{0}^{2A} \rho^{\frac{1-m}{2p-1}}d\rho \right)^{\frac{2p-1}{2p}}<+\infty$.
		Thus (\ref{3.43}) can be written as
		\begin{align}
		   \int_{B(z,2A)\setminus B(z,\zeta)}  \left\{ d^{1-m}\iota  \right\} dv < \psi(\delta|m,p,A).   \label{eqn:OH14_18}
		\end{align}
		Therefore, (\ref{eqn:OH14_16}) follows from the combination of (\ref{3.41}) and (\ref{eqn:OH14_18}). \\		
		
		\noindent
		\textit{Step 5.   Fix $\zeta$ very small. Suppose	 $B(z, 2\zeta) \subset \Omega' \subset \Omega$ and  $\partial \Omega'$ is smooth.  Let $d(\cdot)=d(\cdot, z)$.
		Then we have 
		\begin{align}
		    \left|   \int_{\Omega' \backslash B(z, \zeta)}  \left\{ d^{2-m} \Delta \varphi \right\} dv +\int_{\partial  \Omega'}  \left\{ d^{2-m} \nabla \varphi - \varphi \nabla d^{2-m} \right\} \cdot \vec{n} d\sigma + (m-2)m \varphi(z) \right| < C \zeta^{\alpha} +\psi, 
		    \label{eqn:OH14_22}
		\end{align}
		where $\psi=\psi(\delta|m,p,A)$. 
		}
		
		Applying integration by parts again, we have
		\begin{align*}
		 &\quad \int_{\Omega' \backslash B(z, \zeta)}  \left\{ d^{2-m} \Delta \varphi - \varphi \Delta d^{2-m} \right\} dv\\ 
		 &=-\int_{\partial  \Omega'}  \left\{ d^{2-m} \nabla \varphi - \varphi \nabla d^{2-m} \right\} \cdot \vec{n} d\sigma +\int_{\partial B(z, \zeta)}  \left\{ d^{2-m} \nabla \varphi - \varphi \nabla d^{2-m} \right\} \cdot \vec{n} d\sigma, 
		\end{align*}
		which is equivalent to
		\begin{align}
		 &\quad \int_{\Omega' \backslash B(z, \zeta)}  d^{2-m} \Delta \varphi  dv +\int_{\partial  \Omega'}  \left\{ d^{2-m} \nabla \varphi - \varphi \nabla d^{2-m} \right\} \cdot \vec{n} d\sigma \notag\\
		 &=\underbrace{\int_{\Omega' \backslash B(z, \zeta)}   \varphi \Delta d^{2-m} dv}_{I} + \underbrace{\int_{\partial B(z, \zeta)}  \left\{ d^{2-m} \nabla \varphi - \varphi \nabla d^{2-m} \right\} \cdot \vec{n} d\sigma}_{II}.  
		 \label{eqn:OH14_23}
		\end{align}
		We shall estimate the right hand side of (\ref{eqn:OH14_23}) term by term. 
		Firstly, it follows from (\ref{eqn:OH14_4}) in Step 1 and  (\ref{eqn:OH14_16}) in Step 4 to obtain  that
		\begin{align}
		|I|= \left| \int_{\Omega' \backslash B(z, \zeta)}   \varphi \Delta d^{2-m} dv \right| \leq C \int_{\Omega' \backslash B(z, \zeta)}   |\Delta d^{2-m}| dv \leq \psi(\delta|m,p,A).   \label{eqn:OH14_30} 
		\end{align}
		We move on to estimate $II$ in (\ref{eqn:OH14_23}).	
		Note that
		\begin{align}
		 \int_{\partial B(z, \zeta)}  \left\{ d^{2-m} \nabla \varphi - \varphi \nabla d^{2-m} \right\} \cdot \vec{n} d\sigma
		 =\int_{\partial B(z, \zeta)}  \left\{ d^{2-m} \nabla \varphi (y)- \varphi(y) \nabla d^{2-m} \right\} \cdot \vec{n} d\sigma_{y}.
		\label{eqn:OH14_24} 
		\end{align}
		As $B(z, 2\zeta) \subset \Omega' \subset \Omega$,  we know that 
		\begin{align}
		 &\quad \left|   \int_{\partial B(z, \zeta)}  d^{2-m}  \langle \nabla \varphi (y),  \vec{n} \rangle d\sigma_{y} \right|  \leq  \int_{\partial B(z, \zeta)}  d^{2-m}  |\nabla \varphi| (y) d\sigma_{y} \notag\\
		 &\leq C \zeta^{\alpha-1} \left| \int_{\partial B(z, \zeta)}  d^{2-m} d\sigma_{y} \right| \leq C \zeta^{\alpha-1} \cdot C \zeta \leq C \zeta^{\alpha}. 
		 \label{eqn:OH14_25}
		\end{align}
		On the other hand, we have
		\begin{align}
		&\quad \left| \int_{\partial B(z, \zeta)}  \varphi(y)  \langle \nabla d^{2-m},  \vec{n} \rangle d\sigma_{y}-\int_{\partial B(z, \zeta)}  \varphi(z)  \langle \nabla d^{2-m},  \vec{n} \rangle d\sigma_{y} \right| \notag\\
		&=\left| \int_{\partial B(z, \zeta)}  (\varphi(y)-\varphi(z))  \langle \nabla d^{2-m},  \vec{n} \rangle d\sigma_{y}\right| \leq \int_{\partial B(z, \zeta)}  |\varphi(y)-\varphi(z)| \cdot \left| \langle \nabla d^{2-m},  \vec{n} \rangle \right| d\sigma_{y} \notag\\
		&\leq C \zeta^{\alpha} \cdot \zeta^{1-m} \cdot |\partial B(z, \zeta)| \leq C\zeta^{\alpha}.   \label{eqn:OH14_26}
		\end{align}
		Direct calculation and the volume ratio estimate imply that
		\begin{align}
		 &\quad \left|-m(m-2)\omega_m \varphi(z) + \int_{\partial B(z, \zeta)}  \varphi(z)  \langle \nabla d^{2-m},  \vec{n} \rangle d\sigma_{y} \right| \notag\\
		 &=(m-2)\varphi(z) \left| -m \omega_m + \zeta^{1-m}  |\partial B(z, \zeta)| \right| \leq \psi \cdot \varphi(z),  \label{eqn:OH14_27}
		\end{align}
		where $\psi=\psi(\delta|m,p,A)$. 
		Combining (\ref{eqn:OH14_27}) and (\ref{eqn:OH14_26}), we obtain 
		\begin{align}
		  \left| \int_{\partial B(z, \zeta)}  \varphi(y)  \langle \nabla d^{2-m},  \vec{n} \rangle d\sigma_{y}-m(m-2)\omega_m \varphi(z)  \right| < C \zeta^{\alpha} + \psi(\delta|m,p,A). 
		  \label{eqn:OH14_28}
		\end{align}
		Plugging (\ref{eqn:OH14_28}) and (\ref{eqn:OH14_25}) into (\ref{eqn:OH14_24}) and noting that $\zeta \in (0, 0.1)$,  we obtain
		\begin{align}
		  \left| II + m(m-2) \varphi(z) \right|=\left| \int_{\partial B(z, \zeta)}  \left\{ d^{2-m} \nabla \varphi - \varphi \nabla d^{2-m} \right\} \cdot \vec{n} d\sigma \right| < C \zeta^{\alpha} + \psi(\delta|m,p,A). 
		  \label{eqn:OH14_29}
		\end{align}
		Putting (\ref{eqn:OH14_30}) and (\ref{eqn:OH14_29}) into (\ref{eqn:OH14_23}), we obtain (\ref{eqn:OH14_22}).\\

		\noindent
		\textit{Step 6.   Fix $\zeta \in (0, 0.1)$. Suppose	 $B(z^{*}, 2\zeta) \subset \Omega^{c} \cap B(x_0, \zeta^{-1})$, $\Omega' \subset \Omega$ and  $\partial \Omega'$ is smooth.  Let $d(\cdot)=d(\cdot, z^*)$.
		Then we have 
		\begin{align}
		    \left|   \int_{\Omega'} \left\{ d^{2-m} \Delta \varphi \right\} dv +\int_{\partial  \Omega'}  \left\{ d^{2-m} \nabla \varphi - \varphi \nabla d^{2-m} \right\} \cdot \vec{n} d\sigma \right| < \psi(\delta|m,p).  
		    \label{eqn:OH17_1}
		\end{align}
		}
		
		As $z^*$ is outside of $\Omega'$, integration by parts implies that 
		\begin{align*}
		    \int_{\Omega'} d^{2-m} \Delta \varphi dv +\int_{\partial  \Omega'}  \left\{ d^{2-m} \nabla \varphi - \varphi \nabla d^{2-m} \right\} \cdot \vec{n} d\sigma
		   =\int_{\Omega'}  \varphi \Delta d^{2-m} dv.
		\end{align*}
		In light of (\ref{eqn:OH14_4})  and  (\ref{eqn:OH14_16}), it is clear that (\ref{eqn:OH17_1}) follows directly from the above inequality. \\		
		
		\noindent
		\textit{Step 7.   Fix $\zeta \in (0, 0.1)$. Suppose	 $B(z^{*}, 2\zeta) \subset \Omega^{c} \cap B(x_0, \zeta^{-1})$, $B(z, 2\zeta) \subset \Omega' \subset \Omega$ and  $\partial \Omega'$ is smooth.  Let 
		\begin{align}
		    G(\cdot) \coloneqq  \frac{1}{m(m-2) \omega_m} \left\{ d^{2-m}(z, \cdot) - L d^{2-m}(z^*, \cdot) \right\}.        \label{eqn:OH17_2}
		\end{align}
		Then we have 
		\begin{align}
		    \left|   \int_{\Omega' \backslash B(z,\zeta)}  G \Delta \varphi dv +\int_{\partial  \Omega'}  \left\{ G \nabla \varphi - \varphi \nabla G \right\} \cdot \vec{n} d\sigma + \varphi(z) \right| 
		    < C\zeta^{\alpha} + \psi(\delta|m,p,A).  
		    \label{eqn:OH17_1a}
		\end{align}
		}
		
		It is clear that (\ref{eqn:OH17_1a}) follows from the combination of (\ref{eqn:OH14_22}) and (\ref{eqn:OH17_1}). \\

		\noindent
		\textit{Step 8.  	As $\delta \to 0$,  the limit space $(\R^m, x_{\infty}, g_{E})$ admits a limit function $\varphi_{\infty}$, which is supported on $B(x_{\infty}, A)$ and satisfies the normalization condition.
		Furthermore, $\varphi_{\infty}$ satisfies the integration equation 
		\begin{align}
			\varphi_\infty(z)=\int_{B(x_\infty,A)}G_{\infty}(z,y)\left( \frac{m+m\log\sqrt{4\pi}+\boldsymbol{\bar{\mu}}_\infty}{4}+\frac{\log\varphi_\infty(y)}{2}\right)\varphi_\infty dv_{y} 
		\label{eqn:OH18_6}
		\end{align}
		for any $z\in B(x_\infty,A)$ where $G_{\infty}$ is the Green function of the Euclidean ball $B(x_\infty,A)$:
		\begin{align}
		G_{\infty}(z,y)=\frac{1}{(m-2)m\omega_m}\left( d^{2-m}(z,y)-A^{m-2}d^{2-m}(x_\infty,z)d^{2-m}(z^*,y)\right)
		\label{eqn:OH18_7}
		\end{align}
		whenever $z\neq y$ and $z^*$ is the symmetric point of $z$ with respect to $\partial B(x_\infty,A)$.  
		}

		By continuity, it suffices to show (\ref{eqn:OH18_6}) for all $z\in B(x_\infty,A)\setminus\left\lbrace x_\infty\right\rbrace$. Fix an arbitrary $z\in B(x_\infty,A)\setminus\left\lbrace x_\infty\right\rbrace$ and suppose $z_k\in B(x_k, A)$ 
		and $z_k\longrightarrow z,z_k^*\longrightarrow z^*$ where $z_k^*\in M_k$.
		Define
		\begin{align}
		  &L_{k} \coloneqq A^{m-2} d^{2-m}(x_k, z_k),  \label{eqn:OH17_4}\\
		  &G_{k}  \coloneqq \frac{1}{m(m-2) \omega_m} \left\{ d^{2-m}(z, \cdot) - L_{k} d^{2-m}(z^*, \cdot) \right\}.     \label{eqn:OH17_5}
		\end{align}
		Choose $\Omega_k'$ such that
		\begin{align}
		   &B(x_k, A-\zeta) \subset \Omega_k' \subset B(x_k, A), \label{eqn:OH17_6}\\
		   &\left| \partial B(x_k, A-\zeta) \right| \leq C A^{m-1},   \label{eqn:OH17_7}
		\end{align}
		and $\partial \Omega_k'$ is smooth. 
		It follows from (\ref{eqn:OH17_1a}) that
		\begin{align}
		    &\quad \left|   \int_{\Omega_{k}' \backslash B(z_{k},\zeta)} G_{k} \Delta \varphi_{k} dv +\int_{\partial  \Omega_{k}'}  \left\{ G_{k} \nabla \varphi_{k} - \varphi_{k} \nabla G_{k} \right\} \cdot \vec{n} d\sigma + \varphi_{k}(z_{k}) \right| 
		    < C\zeta^{\alpha},  
		    \label{eqn:OH17_3}
		\end{align}
		whenever $k$ is sufficiently large. 
		
	        Since $z^*$ is the symmetric point of $z \in B(x_{\infty}, A) \backslash x_{\infty}$, it is clear that $G_{\infty}(z, \cdot) \equiv 0$ on $\partial B(x_{\infty}, A)$. 
	        By Lipschitz convergence of $G_k$ to $G_{\infty}$ and the condition (\ref{eqn:OH17_6}), we have
	        \begin{align}
	        &\quad \left| \int_{\partial  \Omega_{k}'}  \left\{ G_{k} \nabla \varphi_{k} - \varphi_{k} \nabla G_{k} \right\} \cdot \vec{n} d\sigma\right| \notag\\
	        &\leq   \int_{\partial  \Omega_{k}'}   G_{k} |\nabla \varphi_{k}| + \varphi_{k} |\nabla G_{k}|  d\sigma
	           \leq   \int_{\partial \Omega_{k}'}   \left\{ C' \zeta \cdot C \zeta^{\alpha-1} + C \zeta^{\alpha} \cdot C' \right\} d\sigma 
	          \leq C' \zeta^{\alpha},  \label{eqn:OH17_8}
	        \end{align}
	        where $C'=C'(\zeta, d(z, x_{\infty}), m,p)$. 
	        Plugging (\ref{eqn:OH17_8}) into (\ref{eqn:OH17_3}) implies that
	        \begin{align}
		    \left|   \int_{\Omega_{k}' \backslash B(z_{k},\zeta)} G_{k} \Delta \varphi_{k} dv+ \varphi_{k}(z_{k}) \right| < C' \zeta^{\alpha}.  
		    \label{eqn:OH17_9}
		\end{align}
		Plugging the proper forms of (\ref{eqn:OH14_9}) and (\ref{eqn:OH14_10}) into (\ref{eqn:OH17_9}) implies that
		 \begin{align}
		    \left|  -\int_{\Omega_{k}' \backslash B(z_{k},\zeta)} G_{k} \left\{ \frac{1}{2} \varphi_{k} \log \varphi_{k} +\boldsymbol{\bar{\lambda}}_{k} \varphi_{k} \right\}  dv+ \varphi_{k}(z_{k}) \right| < C' \zeta^{\alpha}. 
		    \label{eqn:OH17_11}
		\end{align}
		In light of (\ref{eqn:OH14_9}) and (\ref{eqn:OH14_4}), we know that $\boldsymbol{\bar{\lambda}}_{k}$ is uniformly bounded. 
		Due to (\ref{eqn:OH14_4}), the term $\varphi_{k} \log \varphi_{k}$ is uniformly bounded.  
		Therefore, it follows from (\ref{eqn:OH17_11}) that
		\begin{align*}
		 &\quad \left|  -\int_{\Omega_{k}} G_{k} \left\{ \frac{1}{2} \varphi_{k} \log \varphi_{k} +\boldsymbol{\bar{\lambda}}_{k} \varphi_{k} \right\}  dv+ \varphi_{k}(z_{k}) \right| \\
		 &<\left| \int_{B(z_{k},\zeta)} G_{k} \left\{ \frac{1}{2} \varphi_{k} \log \varphi_{k} +\boldsymbol{\bar{\lambda}}_{k} \varphi_{k} \right\}  dv \right| 
		    +\left| \int_{\Omega_{k} \backslash \Omega_{k}'} G_{k} \left\{ \frac{1}{2} \varphi_{k} \log \varphi_{k} +\boldsymbol{\bar{\lambda}}_{k} \varphi_{k} \right\}  dv \right| + C' \zeta^{\alpha}\\
		 &< C \left\{  \int_{B(z_k, \zeta)} G_{k} dv  + \int_{\Omega_{k} \backslash \Omega_{k}'} G_{k} dv \right\} + C' \zeta^{\alpha}\\
		 &< C \int_{B(z_k, \zeta)} \left\{d^{2-m}(\cdot, z_k) + C' \right\} dv + C'|\Omega_{k} \backslash \Omega_{k}'| + C' \zeta^{\alpha} < \left\{C \zeta^{2}+ C'\zeta^{m} \right\} + C' \zeta + C'\zeta^{\alpha}, 
		\end{align*}
		where $C'=C'(\zeta, d(z, x_{\infty}), m,p)$.  Note that we have used the volume comparison in the last step.  Since $\zeta \in (0, 0.1)$, we can combine the last terms and arrive
		\begin{align*}
		  \left|  -\int_{\Omega_{k}} G_{k} \left\{ \frac{1}{2} \varphi_{k} \log \varphi_{k} +\boldsymbol{\bar{\lambda}}_{k} \varphi_{k} \right\}  dv+ \varphi_{k}(z_{k}) \right|
		  <C' \zeta^{\alpha}. 
		\end{align*}
		By volume convergence, the gradient estimate and the boundary $C^{\alpha}$-estimate of $\varphi_{k}$, we can take limit of the above inequality and obtain
		\begin{align*}
		  \left|  -\int_{B(x_{\infty}, A)} G_{\infty} \left\{ \frac{1}{2} \varphi_{\infty} \log \varphi_{\infty} +\boldsymbol{\bar{\lambda}}_{\infty} \varphi_{\infty} \right\}  dv+ \varphi_{\infty}(z) \right| < C' \zeta^{\alpha}.
		\end{align*}
		Putting the corresponding formula of $G_{\infty}$ (cf. (\ref{eqn:OH17_4}) and (\ref{eqn:OH17_5})) into the above inequality, and letting $\zeta \to 0$, we obtain (\ref{eqn:OH18_6}). \\

		\noindent
		\textit{Step 9.  	The limit function $\varphi_{\infty}$ is a strictly positive smooth function on $B(x_{\infty}, A)$ and satisfies
		\begin{align}
			-\Delta \varphi_{\infty} -\frac{1}{2} \varphi_{\infty} \log \varphi_{\infty} =\boldsymbol{\bar{\lambda}}_{\infty} \varphi_{\infty},    \label{eqn:OH18_5}
		\end{align}
		where
		\begin{align}
		  \boldsymbol{\bar{\lambda}}_{\infty}=\frac{m\left(1+\log\sqrt{4\pi}\right) + \boldsymbol{\bar{\mu}}_{\infty}}{4}.  \label{eqn:OH18_4}
		\end{align}
		Furthermore, we have
		\begin{align}
		&\int_{B(x_{\infty}, A)}  \varphi_{\infty}^2 dv=1, \label{eqn:OH18_11}\\
		&\int_{B(x_{\infty}, A)} |\nabla \varphi_{\infty}|^2 dv=\int_{B(x_{\infty}, A)}  \varphi_{\infty} \left\{ \frac{1}{2} \varphi_{\infty} \log \varphi_{\infty}+\boldsymbol{\bar{\lambda}}_\infty \varphi_{\infty} \right\} dv. 
		     \label{eqn:OH17_13}
		\end{align}
		}
		
		It is clear that (\ref{eqn:OH18_5}) and (\ref{eqn:OH18_4}) follow from (\ref{eqn:OH18_6}) directly. 
		By standard regularity theory of elliptic PDE, we know that $\varphi_{\infty}$ is a smooth function on $B(x_{\infty}, A)$.
		Consequently, the estimates (\ref{eqn:OH14_4}), (\ref{eqn:OH18_1}), (\ref{eqn:OH14_5}) and (\ref{eqn:OH14_14}) also hold true by $\varphi_{\infty}$.  In short, we have the estimates
		\begin{align}
		&\varphi_{\infty} \leq  C \min\{ 1, d^{\alpha}(z, \partial B(x_{\infty}, A))\}, \label{eqn:OH18_8} \\
		&|\nabla \varphi_{\infty}|(z) \leq C \{ 1+ d^{\alpha-1}(z, \partial B(x_{\infty}, A))\}.   \label{eqn:OH18_9}
		\end{align}
		As a limit of nonnegative functions $\varphi_i$, it is clear that $\varphi_{\infty} \geq 0$. 
		Since $\varphi_{\infty}$ solves (\ref{eqn:OH18_5}), its zero set is open (cf. the lemma on page 114 of~\cite{Ro}).
		Therefore, either $\varphi_{\infty} \equiv 0$, or $\varphi_{\infty}>0$ everywhere on $B(x_{\infty}, A)$. 	
		However, together with volume comparison and volume convergence, the estimates (\ref{eqn:OH18_8}) and (\ref{eqn:OH18_9}) guarantee that we can take limit of the normalization condition (\ref{eqn:OH11_1}) to obtain (\ref{eqn:OH18_11}). 
		In particular, $\varphi_{\infty}$ is not the zero function.  Therefore, $\varphi_{\infty}>0$ everywhere on $B(x_{\infty}, A)$. 
		
		We move on to show (\ref{eqn:OH17_13}).		 
		Fix an arbitrary $\epsilon>0$. By Sard's theorem, we can find an $s \in (0, \epsilon)$ such that the level set $\{x|\varphi_{\infty}(x)=s\}$ is a smooth hyper-surface.  Then standard integration by parts implies
		\begin{align*}
		    \int_{\{x| \varphi_{\infty}>s\}} |\nabla \varphi_{\infty}|^2 dv&= -\int_{\{x| \varphi_{\infty}>s\}}  \varphi_{\infty} \Delta \varphi_{\infty} dv + \int_{\{x| \varphi_{\infty}=s\}} \varphi_{\infty} \langle \nabla \varphi_{\infty}, \vec{n} \rangle d\sigma\\
		      &=-\int_{\{x| \varphi_{\infty}>s\}}  \varphi_{\infty} \Delta \varphi_{\infty} dv + s \int_{\{x| \varphi_{\infty}=s\}} \langle \nabla \varphi_{\infty}, \vec{n} \rangle d\sigma\\
		      &=-\int_{\{x| \varphi_{\infty}>s\}}  \varphi_{\infty} \Delta \varphi_{\infty} dv + s \int_{\{x| \varphi_{\infty}>s\}} \Delta \varphi_{\infty}  d\sigma. 
		\end{align*}
		Plugging (\ref{eqn:OH18_5}) into the above equations, we obtain
		\begin{align*}
		  &\quad  \int_{\{x| \varphi_{\infty}>s\}} |\nabla \varphi_{\infty}|^2 dv\\
		  &=\int_{\{x| \varphi_{\infty}>s\}}  \varphi_{\infty} \left\{ \frac{1}{2} \varphi_{\infty} \log \varphi_{\infty}+\boldsymbol{\bar{\lambda}}_\infty \varphi_{\infty} \right\} dv
		    -s \int_{\{x| \varphi_{\infty}>s\}} \left\{ \frac{1}{2} \varphi_{\infty} \log \varphi_{\infty}+\boldsymbol{\bar{\lambda}}_\infty \varphi_{\infty} \right\} d\sigma, 
		\end{align*}
		which implies that 
		\begin{align*}
		&\quad \left|  \int_{\{x| \varphi_{\infty}>s\}} |\nabla \varphi_{\infty}|^2 dv - \int_{\{x| \varphi_{\infty}>s\}}  \varphi_{\infty} \left\{ \frac{1}{2} \varphi_{\infty} \log \varphi_{\infty}+\boldsymbol{\bar{\lambda}}_\infty \varphi_{\infty} \right\} dv\right|\\
		&=s \left| \int_{\{x| \varphi_{\infty}>s\}} \left\{ \frac{1}{2} \varphi_{\infty} \log \varphi_{\infty}+\boldsymbol{\bar{\lambda}}_\infty \varphi_{\infty} \right\} d\sigma \right| \leq C s. 
		\end{align*}
		Note that  $B(x_{\infty}, A)=\{x|\varphi_{\infty}(x)>0\}$ by the strict positivity of $\varphi_{\infty}$ on $B(x_{\infty}, A)$. 
		Since $0<s<\epsilon$, we obtain (\ref{eqn:OH17_13}) by letting $\epsilon \to 0$  in the above inequality. \\

		\noindent
		\textit{Step 10.  	Derive the desired contradiction. 
		}
		
		Note that both $ \left\|\varphi_\infty\right\|_{2, B(x_\infty,A)}$ and $ \left\| \nabla\varphi_\infty\right\|_{2, B(x_\infty,A)}$ are finite by  (\ref{eqn:OH17_13}) and (\ref{eqn:OH18_8}). Thus we have $\varphi_{\infty} \in W^{1,2}(B(x_{\infty}, A))$. 
		In view of estimate (\ref{eqn:OH18_8}),  by setting $\varphi_{\infty}|_{\partial B(x_{\infty}, A)} \equiv 0$,  it is clear that $\varphi_{\infty}$ is continuous up to boundary $\partial B(x_{\infty}, A)$, which is clearly smooth.    
		Applying Theorem 2 of section 5.5 of~\cite{Evans},  we have $\varphi_\infty\in W^{1,2}_0(\R^m)$ by setting  $\varphi_{\infty}|_{\R^{m} \backslash B(x_{\infty}, A)} \equiv 0$.
		It follows from (\ref{eqn:OH17_13}), (\ref{eqn:OH18_11}), (\ref{eqn:OH18_4}) and (\ref{eqn:OH14_3}) that
		\begin{align}
		\int_{\R^m}\left\{ 4\left|\nabla\varphi_\infty \right|^2-2\varphi^2_\infty\log\varphi_\infty-m\left(1+\log\sqrt{4\pi} \right)\varphi^2_\infty \right\} dv=\boldsymbol{\bar{\mu}}_\infty<\frac{-\eta}{2}<0, 
		\end{align}
		which contradicts the Logarithm Sobolev inequality of Euclidean space (cf.~\cite{Gro}), which says the left hand side of the above equation is non-negative. This contradiction establishes the proof of (\ref{eqn:OH14_1}).
	\end{proof}
	
	Next we want to prove the following inequality
	\begin{align}\label{3.55}
		\inf\limits_{\tau\in\left[\frac{1}{2},2\right] }\boldsymbol{\mu}\left( B(x_0,A),g,\tau\right)\geq-\eta.
	\end{align}
	It suffices to show the following lemma:

	\begin{lemma}
	\label{l3.10}
		Suppose $p>\frac{m}{2}$, $C_S$ is the $L^2$-Sobolev constant of $\Omega \subset M$. 
		
		 If $\epsilon_0 \coloneqq\left\{ \int_{\Omega}R_-^p dv \right\}^{\frac{1}{p}} \leq \frac{1}{10C_S}$, 	then for any $\tau \in\left[\frac{1}{2},2\right]$ and $\sigma=\tau e^{-C_S \epsilon_0}$, we have 
		\begin{align}
		\boldsymbol{\mu}(\Omega,g,\tau)-\boldsymbol{\bar{\mu}}(\Omega,g,\sigma)\geq -C(m,C_s)\epsilon_0. 
		\label{eqn:OF30_10}
		\end{align}
	\end{lemma}
	
	\begin{proof}
		Suppose $\varphi$ is a minimizer of $\boldsymbol{\mu}\left( \Omega,g,\tau\right)$. Namely, we have $\boldsymbol{\mu}\left( \Omega,g,\tau\right)=\boldsymbol{\mathcal{W}}(\Omega,g,\varphi,\tau)$ and $\int_{\Omega} \varphi^2 dv=1$. 
		Fix $\sigma>0$, it follows from definition(cf. Definition~\ref{d3.1}) that 		\begin{align}
			&\boldsymbol{\mu}\left( \Omega,g,\tau\right) -\boldsymbol{\bar{\mu}}\left( \Omega,g,\sigma\right)  \notag \\ 
			\geq &\boldsymbol{\mathcal{W}}(\Omega,g,\varphi,\tau)-\boldsymbol{\overline{\mathcal{W}}}(\Omega,g,\varphi,\sigma) \notag\\
			=&-\frac{m}{2}\log\left( \frac{\tau}{\sigma}\right) +\int_{\Omega}\tau R\varphi^2dv+4(\tau-\sigma)\int_{\Omega}|\nabla \varphi|^2dv \notag\\
			\geq &-\frac{m}{2}\log\left( \frac{\tau}{\sigma}\right) -\int_{\Omega}\tau R_-\varphi^2dv
			+4(\tau-\sigma)\int_{\Omega}|\nabla \varphi|^2dv \notag\\
			\geq&-\frac{m}{2}\log\left( \frac{\tau}{\sigma}\right) 
			-\tau \epsilon_0 \left( \int_{\Omega}\varphi^{\frac{2p}{p-1}}dv\right) ^{\frac{p-1}{p}}
			+4(\tau-\sigma)\int_{\Omega}|\nabla \varphi|^2dv,    \label{eqn:OF30_6}
		\end{align}
		where we applied H\"older's inequality and the condition $\epsilon_0=\left\{ \int_{\Omega}R_-^p dv \right\}^{\frac{1}{p}}$ in the last step.

		Since $p>\frac{m}{2}$,  we can choose $\theta=1-\frac{m}{2p} \in (0, 1)$ such that  $\frac{\theta}{1}+\frac{1-\theta}{\frac{m}{m-2}}=\frac{1}{\frac{p}{p-1}}$. 
		Then H\"{o}lder's inequality and Young's inequality yield that
		\begin{align*}
		\left( \int_{\Omega}\varphi ^{\frac{2p}{p-1}}dv\right) ^{\frac{p-1}{p}}
		\leq  \left( \int_{\Omega} \varphi^2 dv\right) ^{\theta}\left( \int_{\Omega}\varphi^{\frac{2m}{m-2}}dv\right) ^{\frac{(1-\theta)(m-2)}{m}} 
		\leq  \theta \int_{\Omega}\varphi^2 dv+(1-\theta)\left( \int_{\Omega}\varphi^{\frac{2m}{m-2}}dv\right) ^{\frac{m-2}{m}}.
		\end{align*}
		Applying the normalization condition and the Sobolev inequality, we arrive at
		\begin{align}
		 \left( \int_{\Omega}\varphi ^{\frac{2p}{p-1}}dv\right) ^{\frac{p-1}{p}} \leq \theta + (1-\theta) C_{S} \int_{\Omega} |\nabla \varphi|^2 dv.  \label{eqn:OF30_7}
		\end{align}
		Putting (\ref{eqn:OF30_7}) into (\ref{eqn:OF30_6}), we obtain		
		\begin{align}
			&\quad \boldsymbol{\mu}(\Omega,g,\tau)-\boldsymbol{\bar{\mu}}(\Omega,g,\sigma) \notag\\
			&\geq -\frac{m}{2}\log\left( \frac{\tau}{\sigma}\right) -\tau \epsilon_0\theta
			+\left\{ 4(\tau-\sigma)-\tau \epsilon_0(1-\theta) C_s)\right\} \int_{\Omega}|\nabla \varphi|^2dv.
			\label{eqn:OF30_5}
		\end{align}
		Plugging the particular value $\sigma=e^{-\epsilon_0 C_S} \tau$ into the above inequality yields that
		\begin{align}
			&\quad \boldsymbol{\mu}(\Omega,g,\tau)-\boldsymbol{\bar{\mu}}(\Omega,g,\sigma) \notag\\
			&\geq -\frac{m}{2}C_{S} \epsilon_0 -\tau \epsilon_0\theta
			+\left\{ 4(1-e^{-\epsilon_0 C_S})-\epsilon_0 C_{S} + \epsilon_0 \theta C_s \right\} \tau \int_{\Omega}|\nabla \varphi|^2dv.
		\label{eqn:OF30_9}	
		\end{align}
		Since $\epsilon_0<\frac{1}{10C_S}$, it is clear that $4(1-e^{-\epsilon_0 C_S})-\epsilon_0 C_{S}>0$.
		Therefore,  the term inside the parenthesis of (\ref{eqn:OF30_9}) is positive. 
		Recall that $\theta=1-\frac{m}{2p} \in (0,1)$ and $\tau \in [\frac{1}{2}, 2]$.  Then (\ref{eqn:OF30_10}) follows directly from (\ref{eqn:OF30_9}). 
		The proof of Lemma~\ref{l3.10} is complete. 
	\end{proof}

	\begin{theorem}
		Same conditions as in Proposition~\ref{prn:OF29_7}. If (\ref{eqn:OF29_5a}) and (\ref{eqn:OF29_5b}) hold,  then we have 
		\begin{align}
		  &\inf\limits_{\tau\in\left[\frac{1}{2},2\right] }\boldsymbol{\mu}\left( B(x_0,A),g,\tau\right)\geq-\eta, \label{3.61} \\
		  &\inf_{t\in(0,1]}\boldsymbol{\mu}\left( B_{g(0)}\left( x_0,A\sqrt{t}\right) ,g(0),t\right) \geq -\eta.  \label{eqn:OF30_1}
		\end{align}
	\label{thm:OF29_8}	
	\end{theorem}
	
	\begin{proof}
	
	        Note that (\ref{eqn:OF30_1}) follows from (\ref{3.61}), in light of scaling invariant property of $\boldsymbol{\mu}$ and the almost monotonicity of  $\kappa(p, r)$ and the volume ratio (cf. Lemma~\ref{l2.4} and Lemma~\ref{l2.5}). 
	        Therefore, it suffices to prove (\ref{3.61}). 
	
		In view of (\ref{eqn:OF29_5a}) and (\ref{eqn:OF29_5b}), we can choose a sufficiently small $\delta$ such that $B(x_0,A)\subset B(x_0,\delta^{-1})$.
		Now we apply Lemma~\ref{l3.10} for $\Omega=B(x_0,A)$ and $\tau\in[\frac{1}{2},2]$.  By (\ref{eqn:OF30_10}), we can choose a $\sigma \in [\frac{1}{3},3]$ such that
		\begin{equation}\label{3.63}
			\boldsymbol{\mu}\left( B(x_0,A),g,\tau\right)-\boldsymbol{\bar{\mu}}\left( B(x_0,A),g,\sigma\right)\geq\frac{-\eta}{3}.
		\end{equation}
		Since $\sigma \in [\frac{1}{3}, 3]$, it follows from (\ref{3.16}) that 
		\begin{equation}\label{3.64}
			\boldsymbol{\bar{\mu}}(B(x_0,A),g,\sigma)\geq\frac{-\eta}{2}.
		\end{equation}
		Combining (\ref{3.63}) and (\ref{3.64}), we obtain (\ref{3.61}).   The proof of Theorem \ref{thm:OF29_8} is complete. 
	\end{proof}
			
	In light of (\ref{eqn:OF30_1}), we can apply Theorem 1.2 of~\cite{WB2} and derive the following theorem. 
	
	\begin{theorem}
		For each $p>\frac{m}{2}$ and $\epsilon$ small, there exists a number $\delta =\delta (m,p,\epsilon)$ satisfying the following properties. 
		
		Suppose $\{(M^m,g(t)),t\in[0,1]\}$ is a solution of Ricci flow.  Suppose under the metric $g(0)$ it holds that
		\begin{subequations}
			\begin{align}[left = \empheqlbrace \,]
				&\int_{B(x_0,\delta^{-1})}Rc_-^p dv\leq \frac{1}{2} \omega_m \delta^{4p-m},   \label{eqn:OF29_6a} \\
				&\left| B\left(x_0, \delta^{-1}\right) \right| \geq(1-\delta) \omega_m\delta^{-m}.    \label{eqn:OF29_6b}
			\end{align}  
		\label{eqn:OF29_6}	
		\end{subequations}
		\hspace{-2mm}
		Then for each $t\in(0,1]$ and $x \in B_{g(t)} \left(x, \epsilon^{-1} \sqrt{t} \right)$,  we have 
		\begin{subequations}
			\begin{align}[left = \empheqlbrace \,]
				&\ t|Rm|(x,t)\leq \epsilon,  \label{eqn:OH27_15a} \\
				&\inf_{\rho\in\left( 0, \epsilon^{-1}\sqrt{t}\right) }\rho^{-m}vol_{g(t)}(B_{g(t)}(x,\rho))\geq(1-\epsilon)\omega_m,  \label{eqn:OH27_15b}\\
				&t^{-\frac{1}{2}}inj(x,t)\geq \epsilon^{-1}.  \label{eqn:OH27_15c}
			\end{align}  
		\label{eqn:OH27_15}	
		\end{subequations}
				
	\label{thm:OF28_5}	
	\end{theorem}
	
	We close this section by the proof of Theorem~\ref{thm:OD12_1}.
	
	\begin{proof}[Proof of Theorem~\ref{thm:OD12_1}:]
	It is clear that (\ref{eqn:OD12_4}) follows directly from Proposition~\ref{prn:OF29_7}.   The inequality (\ref{eqn:OF16_2}) follows from Theorem~\ref{thm:OF29_8}.
	\end{proof}

	\section{Estimate of volume and scalar curvature integral}
	\label{sec:SLP}
	
	Suppose $\{(M^m,g(t)), t \in[0,1]\}$ is a solution of Ricci flow with an initial data $(M,g(0))$. 
	In Section~\ref{sec:functionals}, we prove that if $(M,g(0))$ satisfies the integral Ricci curvature condition (\ref{eqn:OF29_6a}) and the volume lower bound condition (\ref{eqn:OF29_6b}), then we have the pseudo-locality estimate (\ref{eqn:OH27_15}). 
	In this section, we shall provide local estimates of volume and $L^p$-norm of scalar curvature along the Ricci flow. In \cite{WB2}, similar estimates were obtained under the conditions 
	\begin{align*}
	  |Rc|\leq \frac{\epsilon}{t}, \quad \textrm{and} \quad R_{g(0)} \geq -\epsilon, 
	\end{align*}
	with the help of localized maximum principle(cf. Theorem 5.1 of~\cite{WB2}).  In this section, we use the conditions
	\begin{align*}
	 R_{-} \leq \frac{\epsilon}{t}, \quad \textrm{and} \quad \int_{B_{g(0)}(x_0,r_0)} R_-^pdv\leq \epsilon, 
	\end{align*}
	which can be deduced from (\ref{eqn:OF29_6a}) and (\ref{eqn:OF29_6b}) naturally. 
	
	The local  estimates of volume and $L^p$-norm of scalar curvature are key new ingredients of this paper. 
	Before we delve into the details of proving these estimates, let us detour for elementary technical preparations.

	\begin{lemma}
		\label{lem4.1}
		Suppose $\{(M^m,g(t)),  0 \leq t \leq 1\}$ is a smooth flow. $\Omega$ is a compact subset of $M$. $f\in C^{\infty}\left( M\times [0,1]\right) $ and $\eta\in C^{1}\left( M\times [0,1]\right) $. Let $p>1$ and $\Omega_{0}(t)=\left\lbrace x\in \Omega : f(x,t)\geq 0\right\rbrace $. Then we have
		\begin{equation}\label{equ4.1}
			\frac{d}{dt}\int_{\Omega_0}\eta f^pdv
			=\int_{\Omega_0} \frac{d}{dt}\left( \eta f^p\right) dv+\int_{\Omega_0} \eta f^p\left( \frac{d}{dt}dv\right)
		\end{equation}
		for any $t\in [0,1]$.
	\end{lemma}
	
	\begin{proof}
	We remind the reader that the boundary of  $\Omega_0(t)$ may not be smooth, which may cause unexpected difficulty.
	We shall prove \ref{equ4.1} by definition.  Namely, we shall prove the following equation. 
	\begin{align}
	&\quad \lim_{\Delta t \to 0}\frac{1}{\Delta t}\left(\int_{\Omega_{0}(t+\Delta t)}\eta (x,t+\Delta t)f^p(x,t+\Delta t)dv_{t+\Delta t} -\int_{\Omega_{0}(t)}\eta (x,t)f^p(x,t)dv_{t}\right)  \notag\\
	&=\int_{\Omega_0(t)} \eta(x,t) f^p(x,t)\left( \frac{d}{dt}dv_t\right)+\int_{\Omega_0(t)} \frac{d}{dt}\left( \eta (x,t)f^p(x,t)\right) dv_t. \label{equ4.2}
	\end{align}
	In order to prove  (\ref{equ4.2}), it suffices to prove the following three equations. 
	\begin{align}
	 &\lim_{\Delta t \to 0}\frac{1}{\Delta t}\left(\int_{\Omega_{0}(t+\Delta t)}\eta f^p(x,t+\Delta t)dv_{t+\Delta t} -\int_{\Omega_{0}(t)}\eta f^p(x,t+\Delta t)dv_{t+\Delta t}\right)=0, \label{equ4.3}\\
	 &\lim_{\Delta t \to 0}\frac{1}{\Delta t}\left(\int_{\Omega_{0}(t)}\eta f^p(x,t+\Delta t)dv_{t+\Delta t} -\int_{\Omega_{0}(t)}\eta f^p(x,t+\Delta t)dv_{t}\right) 
			=\int_{\Omega_0(t)} \eta f^p(x,t)\left( \frac{d}{dt}dv_t\right), \label{equ4.4}\\
	&\lim_{\Delta t \to 0}\frac{1}{\Delta t}\left(\int_{\Omega_{0}(t)}\eta f^p(x,t+\Delta t)dv_{t} -\int_{\Omega_{0}(t)}\eta f^p(x,t)dv_{t}\right) 
			=\int_{\Omega_0(t)} \frac{d}{dt}\left( \eta f^p(x,t)\right) dv_t. \label{equ4.5}
	\end{align}
	We shall prove them step by step.\\
		
		\textit{Step 1.  Equation (\ref{equ4.3}) holds.}
		
		Note that
			\begin{align}
				&\lim_{\Delta t \to 0}\frac{1}{\Delta t}\left(\int_{\Omega_{0}(t+\Delta t)}\eta f^p(x,t+\Delta t)dv_{t+\Delta t} 
				-\int_{\Omega_{0}(t)}\eta f^p(x,t+\Delta t)dv_{t+\Delta t}\right) \notag\\
				=&\lim_{\Delta t \to 0}\frac{1}{\Delta t}\left(\int_{\Omega_{0}(t+\Delta t)\backslash\Omega_{0}(t)}\eta f^p(x,t+\Delta t)dv_{t+\Delta t} 
				-\int_{\Omega_{0}(t)\backslash\Omega_{0}(t+\Delta t)}\eta f^p(x,t+\Delta t)dv_{t+\Delta t}\right). \label{equ4.6}
			\end{align}
		
		When $\Delta t\rightarrow 0_+$, in $\Omega_{0}(t+\Delta t)\backslash\Omega_{0}(t)$, there holds
		\begin{equation*}
			f(x,t)<0,\quad f(x,t+\Delta t)\geq 0.
		\end{equation*}
		Since $\left| \frac{df}{dt}\right|\leq C $ in $\Omega\times [0,1]$, there holds
		\begin{equation*}
			0 \leq f(x,t+\Delta t)<C\Delta t
		\end{equation*}
		in $\Omega_{0}(t+\Delta t)\backslash\Omega_{0}(t)$. 
		
		Similarly, in $\Omega_{0}(t)\backslash\Omega_{0}(t+\Delta t)$, there holds
		\begin{equation*}
			f(x,t)\geq 0,\quad f(x,t+\Delta t)< 0, 
		\end{equation*}
		which implies that
		\begin{equation*}
			-C\Delta t \leq f(x,t+\Delta t)<0
		\end{equation*}
		in $ \Omega_{0}(t+\Delta t)\backslash\Omega_{0}(t)$. Thus we have
			\begin{align*}
				&\frac{1}{\Delta t}\left|\int_{\Omega_{0}(t+\Delta t)\backslash\Omega_{0}(t)}\eta f^p(x,t+\Delta t)dv_{t+\Delta t} 
				-\int_{\Omega_{0}(t)\backslash\Omega_{0}(t+\Delta t)}\eta f^p(x,t+\Delta t)dv_{t+\Delta t}\right| \notag\\
				\leq &\frac{1}{\Delta t}\left|\int_{\Omega_{0}(t+\Delta t)\backslash\Omega_{0}(t)}\eta f^p(x,t+\Delta t)dv_{t+\Delta t} \right| 
				+\frac{1}{\Delta t}\left| \int_{\Omega_{0}(t)\backslash\Omega_{0}(t+\Delta t)}\eta f^p(x,t+\Delta t)dv_{t+\Delta t}\right| \notag\\
				\leq &\frac{1}{\Delta t}CV(\Delta t)^p 
				\leq  C(\Delta t)^{p-1}, 
			\end{align*}
		where $V$ is the upper bound of $vol_{g(t)}\Omega$ when $t\in [0,1]$.  As $p>1$, it is clear that
		\begin{equation}\label{equ4.8}
			\lim_{\Delta t \to 0_+}\frac{1}{\Delta t}\left(\int_{\Omega_{0}(t+\Delta t)}\eta f^p(x,t+\Delta t)dv_{t+\Delta t} 
			-\int_{\Omega_{0}(t)}\eta f^p(x,t+\Delta t)dv_{t+\Delta t}\right)=0.
		\end{equation}
		Similar deduction yields that 
		\begin{equation}\label{equ4.9}
			\lim_{\Delta t \to 0_-}\frac{1}{\Delta t}\left(\int_{\Omega_{0}(t+\Delta t)}\eta f^p(x,t+\Delta t)dv_{t+\Delta t} 
			-\int_{\Omega_{0}(t)}\eta f^p(x,t+\Delta t)dv_{t+\Delta t}\right)=0.
		\end{equation}
		Therefore, (\ref{equ4.3}) follows from the combination of (\ref{equ4.8}) and (\ref{equ4.9}). \\
		
		\textit{Step 2.  Equation (\ref{equ4.4}) holds.}
		
		Since $\frac{d}{dt} \log dv_t=\frac{1}{2} tr_{g} \dot{g}$  and  $\{(M^m,g(t)),  0 \leq t \leq 1\}$ is a smooth flow, it is clear that 
		\begin{align*}
			dv_{t+\Delta t}-dv_{t}=\left( \frac{1}{2} tr_{g} \dot{g} \Delta t+o(\Delta t)\right) dv_{t}.
		\end{align*}
		Consequently, by continuity of $\eta f^p$, there holds that
			\begin{align*}
				&\lim_{\Delta t \to 0}\frac{1}{\Delta t}\left(\int_{\Omega_{0}(t)}\eta f^p(x,t+\Delta t)dv_{t+\Delta t} 
				-\int_{\Omega_{0}(t)}\eta f^p(x,t+\Delta t)dv_{t}\right) \notag\\
				=&\lim_{\Delta t \to 0}\frac{1}{\Delta t}\int_{\Omega_{0}(t)}\eta f^p(x,t+\Delta t)\left(  \frac{1}{2} tr_{g} \dot{g}  \Delta t+o(\Delta t)\right)dv_{t} \notag\\
				=&\int_{\Omega_0(t)} \eta f^p(x,t)\left( \frac{d}{dt}dv_t\right), 
			\end{align*}
		which is nothing but  (\ref{equ4.4}). \\
		
		It is obvious that the equation (\ref{equ4.5}) holds.  Therefore,  equation (\ref{equ4.2})  follows immediately from the combination of  (\ref{equ4.3}), (\ref{equ4.4}) and (\ref{equ4.5}).
	\end{proof}
	
	\begin{lemma}
		\label{lem4.2}
		Suppose $M^m$ is a smooth manifold. $\Omega \subset M$ is a bounded domain,  $f\in C^{\infty}\left( M\right) $ and $\eta\in C^{1}\left( M\right) $. 
		Let $p>1$ and $\Omega_{h} \coloneqq \left\lbrace x\in \Omega : f(x)\geq h\right\rbrace $. Then we have
		\begin{equation}\label{equ4.12}
			\int_{\Omega_0}\eta \Delta f^pdv=-\int_{\Omega_0}\langle\nabla\eta, \nabla f^p\rangle dv.
		\end{equation}
	\end{lemma}
	
	\begin{proof}
	We first assume $\Omega$ has smooth boundary. 
	
	Recall that $\Omega_0 = \left\lbrace x\in \Omega : f(x)\geq 0\right\rbrace$. 
	If  $\partial \Omega_0$ is smooth, then equation (\ref{equ4.12}) holds obviously.
	Otherwise, we shall show (\ref{equ4.12}) by approximation. 
	Suppose  $h$ is a regular value of $f$, then $\partial \Omega_h$ is smooth. 
	Integration by parts yields that
		\begin{equation}\label{equ4.13}
			\int_{\Omega_{h}}\eta \Delta f^pdv+\int_{\Omega_{h}}\langle\nabla\eta, \nabla f^p\rangle dv
			=\int_{\partial\Omega_{h}}\langle\eta\nabla f^p,\vec{n}\rangle d\sigma=\int_{\partial\Omega_{h}}\langle\eta f^{p-1}\nabla f,\vec{n}\rangle d\sigma.
		\end{equation}
		Since $f=h$ on $\partial \Omega_{h}$, there holds 
		\begin{equation}\label{equ4.14}
			\int_{\partial\Omega_{h}}\langle\eta f^{p-1}\nabla f,\vec{n}\rangle d\sigma=h^{p-1}\int_{\partial\Omega_{h}}\langle\eta \nabla f,\vec{n}\rangle d\sigma
			=h^{p-1}\int_{\Omega_{h}} \left\{ \eta \Delta f+\langle\nabla\eta, \nabla f\rangle \right\} dv.
		\end{equation}
		Consequently we have
		\begin{equation}\label{equ4.15}
			\left| \int_{\Omega_{h}}\left\{ \eta \Delta f^p+\langle\nabla\eta, \nabla f^p\rangle\right\} dv\right| \leq h^{p-1}Cvol(\Omega).
		\end{equation}
		In $\Omega_{0}\backslash \Omega_{h}$, since $0\leq f\leq h$ , there holds
		\begin{equation}\label{equ4.16}
			\langle\nabla\eta, \nabla f^p\rangle+\eta \Delta f^p =pf^{p-1}\langle\nabla\eta, \nabla f\rangle+\eta \left(p(p-1)f^{p-2}|\nabla f|^2+pf^{p-1}\Delta f \right)\geq -Ch^{p-1}.
		\end{equation}
		By Sard's theorem, we can choose regular values $h \to 0$.  It follows from (\ref{equ4.15}) and (\ref{equ4.16}) that  $\int_{\Omega_0}\eta \Delta f^p dv+\int_{\Omega_0}\langle\nabla\eta, \nabla f^p\rangle dv$ exists and
		\begin{align*}
			\int_{\Omega_0}\eta \Delta f^p dv+\int_{\Omega_0}\langle\nabla\eta, \nabla f^p\rangle dv=\lim_{h\to 0}\left( \int_{\Omega_{h}}\eta \Delta f^p dv+\int_{\Omega_{h}}\langle\nabla\eta, \nabla f^p\rangle dv\right) =0, 
		\end{align*}
		which is equivalent to (\ref{equ4.12}). 
		
		Now we consider the general situation. Note that $\Omega$ can be exhausted by a sequence of domains $\tilde{\Omega}_k$, which have smooth boundaries (cf. the proof of Lemma 2.6 of~\cite{WB}).
		Namely, we have
		\begin{align*}
		   \tilde{\Omega}_1 \subset \tilde{\Omega}_2 \subset \cdots \tilde{\Omega}_k \subset \tilde{\Omega}_{k+1} \subset \cdots,  \quad \Omega=\cup_{k=1}^{\infty} \tilde{\Omega}_k. 
		\end{align*}
		From previous argument, we know that 
		 \begin{align*}
			\int_{\Omega_0 \cap \tilde{\Omega}_k}\eta \Delta f^p dv=-\int_{\Omega_0 \cap \tilde{\Omega}_k} \langle\nabla\eta, \nabla f^p\rangle dv.
		\end{align*}
		Since $\left\{ \Omega_0 \cap \tilde{\Omega}_k \right\}_{k=1}^{\infty}$ is an exhaustion of $\Omega_0$, we can take limit of both sides of the above equation to obtain (\ref{equ4.12}).
		The proof of Lemma~\ref{lem4.2} is complete.  
	\end{proof}

	Now we are ready to prove the main theorem of this section. 
			
	\begin{theorem}
	\label{t4.2}
		Suppose $\{(M^m,g(t)),  0 \leq t \leq 1\}$ is a Ricci flow solution. Let $r_0\geq 3$ and $p>\frac{m}{2}$. 
		Suppose
		\begin{subequations}
			\begin{align}[left = \empheqlbrace \,]
				&vol_{g(0)} \left( B_{g(0)}(x_0,r_0) \right) \leq V_0,   \label{eqn:OH12_26a} \\
				&\int_{B_{g(0)}(x_0,r_0)} R_{-,g(0)}^p dv\leq \epsilon,  \label{eqn:OH12_26b}
			\end{align}  
		\end{subequations}
		and 
		\begin{align}
			R_{-}(x,t)\leq \frac{\epsilon}{t}, \quad  \forall \; x \in B_{g(0)}(x_0,r_0), \; t \in (0, 1].        \label{eqn:OH12_3}
		\end{align}
		Then for each $t\in [0,1]$ we have
		\begin{subequations}
			\begin{align}[left = \empheqlbrace \,]
				&vol_{g(t)}\left(  B_{g(0)}(x_0,r_0-2)\right) \leq V_0+\frac{(V_0+\epsilon)\left( e^{2Ct}-1\right) }{2}, \label{4.9}  \\
				&\int_{B_{g(0)}(x_0,r_0-2)} R_{-,g(t)}^p dv\leq \epsilon+\frac{(V_0+\epsilon)\left( e^{2Ct}-1\right) }{2}. \label{4.8}
			\end{align}  
		\end{subequations}
	\end{theorem}
	
	\begin{proof}
	        The proof consists of five steps. We shall first construct a cutoff function $\eta$ with proper properties in Step 1.
	        Then we calculate the evolution of $\int_{M} \eta dv$ and $\int_{M} \eta^{\frac{3}{2}} R_{-}^{p} dv$ in Step 2 and Step 3, respectively.
	        In Step 4, we dominate $\int_{M} \eta dv$ and $\int_{M} \eta^{\frac{3}{2}} R_{-}^{p} dv$ by ODE solutions which 
	        can be calculated explicitly.   Finally, in Step 5, we  focus on the domain where $\eta \equiv 1$ and  finish the proof of the estimate (\ref{4.9}) and (\ref{4.8}). 
		All constants $C$ in this proof depend only on $m$ and $p$ and may vary from line to line. \\
		
		\noindent
		\textit{Step 1.  Construction of a proper cutoff function $\eta$.}

		Define a cut-off function $\eta :\mathbb{R}\rightarrow\mathbb{R}$ satisfying $\eta'\leq 0$ and 
		\begin{align}
		\eta(s)=
		\begin{cases}
		0,  & s \in[r_0,\infty);\\
		\left(\frac{1}{4p+4}(r_0-x)\right)^{4p+4} ,  & s \in\left[ r_0-\frac{1}{3},r_0\right] ; \\
		1 ,  & s \in (-\infty, r_0-1] .
		\end{cases}
		\label{eqn:OH12_11}
		\end{align}
		Furthermore,  we can choose an $\eta \in C^{3}$ such that
		\begin{align}
		&\eta'=-\eta^{\frac{4p+3}{4p+4}}, \quad \textrm{if} \; s \in\left[ r_0-\frac{1}{3},r_0\right];  \label{eqn:OH12_12}\\
		&-\eta''\leq C\eta, \quad \eta^{\frac{2p-3}{2p-2}} \leq C\eta, \quad 0 \leq  -\eta' \leq C\eta^\frac{4p+3}{4p+4}, \quad \textrm{if} \; s \in\left[ 0,r_0-\frac{1}{3}\right]. \label{eqn:OH12_13}
		\end{align}
		Abusing notation, we define
		\begin{align}
		&\alpha \coloneqq \frac{p-1}{2p+2},  \label{eqn:OH12_1}\\
		&\eta \coloneqq \eta(d_{g(0)}(x_0,x)+t^\alpha).   \label{eqn:OH12_2}
		 \end{align}
		
		\vspace{4mm}
		
		\noindent
		\textit{Step 2.  It holds that 
		\begin{align}
		\frac{d}{dt}\int_M \eta dv \leq C \left\{  \int_M \eta dv+ \int_M \eta^\frac{3}{2}R_{-}^p dv \right\}.    \label{eqn:OH12_6}
		\end{align}
		}
		
		Direct calculation shows that
		\begin{align}
		\frac{d}{dt}\int_M \eta dv 
		=\int_M \left( \frac{d}{dt}\eta\right)  dv+\int_M \eta \left( \frac{d}{dt}dv\right)  
		=\int_M \eta'\alpha t^{\alpha-1}dv+\int_M \eta (-R)dv. 
		\label{eqn:OH12_4}
		\end{align}
		Since $\eta' \leq 0$ and $\alpha \in (0, 1)$,  it follows from (\ref{eqn:OH12_3}) and (\ref{eqn:OH12_1}) that 
		\begin{align}
		   \alpha \eta't^{\alpha-1}\leq  \alpha \eta'\frac{R_-^{1-\alpha}}{\epsilon^{1-\alpha}}=\frac{\alpha}{\epsilon^{1-\alpha}} \cdot \eta'  \cdot R_{-}^{1-\alpha}  \leq \eta' R_{-}^{1-\alpha}= \eta' R_{-}^{\frac{p+3}{2p+2}}.   \label{eqn:OH12_5}
		\end{align}
		Plugging (\ref{eqn:OH12_5}) into (\ref{eqn:OH12_4}) and noting that $-R \leq R_{-}$,  we obtain 
		\begin{align}
		\frac{d}{dt}\int_M \eta dv 	\leq \int_M \eta'R_-^{\frac{p+3}{2p+2}}dv+\int_M \eta R_{-} dv.   \label{eqn:OH12_7}
		\end{align}
		Now we divide the support of $\eta$ into two parts:
		\begin{align*}
		 &M_1 \coloneqq \left\lbrace x\in M:d_{g(0)}(x_0,x)+t^\alpha \in\left[ 0,r_0-\frac{1}{3}\right] \right\rbrace, \\
		 &M_2 \coloneqq \left\lbrace x\in M:d_{g(0)}(x_0,x)+t^\alpha \in\left[ r_0-\frac{1}{3},r_0\right] \right\rbrace. 
		\end{align*}
		It follows from (\ref{eqn:OH12_12}) that $\eta'=-\eta^{\frac{4p+3}{4p+4}}$ on $M_2$.  Then (\ref{eqn:OH12_7}) can be written as
		\begin{align}
		\frac{d}{dt}\int_M \eta dv 
		&\leq \int_{M_1} \eta'R_-^{\frac{p+3}{2p+2}}dv+\int_{M_2} \eta'R_-^{\frac{p+3}{2p+2}}dv+\int_{M_1} \eta R_-dv+ \int_{M_2} \eta R_-dv \notag\\
		&\leq \int_{M_2} -\eta^{\frac{4p+3}{4p+4}}R_-^{\frac{p+3}{2p+2}}dv+\int_{M_1} \eta R_-dv+\int_{M_2} \eta R_-dv.   \label{eqn:OH12_10}
		\end{align}
		H\"{o}lder's inequality and Young's inequality imply that
		\begin{align}
		\int_{M_1} \eta R_-dv&=\int_{M_1} (\eta^{\frac{3}{2p}} R_{-}) \cdot \eta^{\frac{2p-3}{2p}} dv
		\leq \bigg(\int_{M_1} \eta^\frac{3}{2}R_-^p dv\bigg)^{\frac{1}{p}} \bigg(\int_{M_1} \eta^{\frac{2p-3}{2p-2}}dv\bigg)^{\frac{p-1}{p}} \notag\\
		&\leq C\int_{M_1} \eta^\frac{3}{2}R_-^p dv+\int_{M_1} \eta^{\frac{2p-3}{2p-2}}dv.   \label{eqn:OH12_8}
		\end{align}
		Similarly, we have
		\begin{align}
		\int_{M_2} \eta R_-dv&=\int_{M_2} \left(\eta^{\frac{3}{4p+6}} R_{-}^{\frac{p}{2p+3}} \right) \cdot \left( \eta^{\frac{4p+3}{4p+6}} R_{-}^{\frac{p+3}{2p+3}} \right) dv  \notag\\
		&\leq \left(\int_{M_2} \eta^\frac{3}{2}R_-^p dv\right)^{\frac{1}{2p+3}} 
		   \left(\int_{M_2} \eta^{\frac{4p+3}{4p+4}}R_-^{\frac{p+3}{2p+2}}dv \right)^{\frac{2p+2}{2p+3}} \notag\\
		&\leq C\int_{M_2} \eta^\frac{3}{2}R_-^p dv+\int_{M_2} \eta^{\frac{4p+3}{4p+4}}R_-^{\frac{p+3}{2p+2}}dv.  \label{eqn:OH12_9}
		\end{align}
		Plugging (\ref{eqn:OH12_8}) and (\ref{eqn:OH12_9}) into (\ref{eqn:OH12_10}) and applying (\ref{eqn:OH12_13}),  we obtain
		\begin{align}
			\frac{d}{dt}\int_M \eta dv &\leq C\int_M \eta^\frac{3}{2}R_-^p dv+\int_{M_1} \eta^{\frac{2p-3}{2p-2}}dv 
			\leq C \left\{  \int_M \eta dv+ \int_M \eta^\frac{3}{2}R_{-}^p dv \right\}, 
		\end{align}
		which is nothing but (\ref{eqn:OH12_6}). \\

		\noindent	
		\textit{Step 3.  It holds that 
		\begin{align}
		   \frac{d}{dt} \int_M \eta^\frac{3}{2}R_{-}^pdv  \leq C \left\{  \int_M \eta dv+ \int_M \eta^\frac{3}{2}R_{-}^p dv \right\}.    \label{eqn:OH12_14}
		\end{align}
		}

		For each time $t$, we set 
		\begin{align*}
		  M_{-}(t) \coloneqq \{x\in M:R(x,t)\leq 0\}. 
		\end{align*}
		Then $R_{-} \equiv 0$ on $M \backslash M_{-}(t)$ and 
		\begin{align*}
		  \int_M\eta^\frac{3}{2}R_-^pdv=\int_{M_-}\eta^\frac{3}{2}(-R)^p dv. 
		\end{align*}
        From Lemma \ref{lem4.1}, there holds
		\begin{align}
			&\quad \frac{d}{dt}\int_M \eta^\frac{3}{2}R_-^p dv 
			   = \frac{d}{dt}\int_{M_-}\eta^\frac{3}{2}(-R)^p dv \notag\\
			&= \int_{M_-}\left( \frac{d}{dt}\eta^\frac{3}{2}\right) (-R)^pdv+\int_{M_-} \eta^\frac{3}{2} p(-R)^{p-1}\frac{d}{dt}(-R)dv 
			+\int_{M_-} \eta^\frac{3}{2}(-R)^p\left( \frac{d}{dt}dv\right). 
	         \label{eqn:OH12_15}		
		\end{align}
		Along the Ricci flow, we have
		\begin{align*}
		  &\frac{d}{dt}\eta^\frac{3}{2}=\frac{3}{2}\eta^{\frac{1}{2}}\eta'\alpha t^{\alpha-1}\leq 0, \\
		  &\frac{d}{dt} dv= - R dv, \\
		  &\frac{d}{dt} R=\Delta R + 2|Rc|^2 \geq \Delta R +\frac{2R^2}{m}. 
		\end{align*}
		Putting them into (\ref{eqn:OH12_15}) and noting that $p>\frac{m}{2}$, we obtain
		\begin{align}
		   \frac{d}{dt}\int_M \eta^\frac{3}{2}R_-^p dv 
		   \leq \int_{M_-} \eta^\frac{3}{2} p(-R)^{p-1}\Delta (-R)dv-\left( \frac{2p}{m}-1\right) \int_{M_-} \eta^\frac{3}{2}(-R)^{p+1}dv.  
		   \label{eqn:OH12_16}   
		\end{align}
		Note that
		\begin{align*}
		     \Delta (-R)^p=p(p-1) (-R)^{p-2}|\nabla (-R)|^2+p(-R)^{p-1}\Delta (-R). 
		\end{align*}
		Plugging the above equation into (\ref{eqn:OH12_16}) yields that 
		\begin{align}
			&\quad \frac{d}{dt}\int_M \eta^\frac{3}{2}R_-^pdv \notag\\ 
			&\leq \int_{M_-} \eta^\frac{3}{2} \Delta (-R)^pdv-\int_{M_-} \eta^\frac{3}{2} p(p-1) (-R)^{p-2}|\nabla (-R)|^2dv
			-\left( \frac{2p}{m}-1\right) \int_{M_-} \eta^\frac{3}{2}(-R)^{p+1}dv.  
			\label{eqn:OH12_19}
		\end{align}
		Since $(-R)^p=0$ and $\nabla (-R)^p=0$ on the boundary of $M_{-}(t)$, it follows from Lemma \ref{lem4.2} and elementary inequality that
		\begin{align}
		  &\quad \int_{M_-} \eta^\frac{3}{2} \Delta (-R)^p dv \notag\\
		  &=\int_{M_-} p(-R)^{p-1} \left\langle \frac{3}{2} \eta^{\frac{1}{2}} \nabla \eta, \nabla (-R) \right\rangle dv
		    \leq \frac{3}{2}\int_{M_-} \eta^{\frac{1}{2}}|\nabla\eta|p(-R)^{p-1}|\nabla (-R)|dv \notag\\
		  &\leq \int_{M_-} \eta^\frac{3}{2} p(p-1) (-R)^{p-2}|\nabla (-R)|^2dv + \frac{9p}{16(p-1)}\int_{M_-} \frac{|\nabla\eta|^2}{\eta^{\frac{1}{2}}}(-R)^p dv.   \label{eqn:OH12_17}
		\end{align}
		In light of the choice of $\eta$ in (\ref{eqn:OH12_12}) and (\ref{eqn:OH12_13}), it is clear that $|\nabla\eta|^2\leq |\eta'|^2\leq C\eta^{\frac{4p+3}{2p+2}}$.
		Thus 
		\begin{align}
		  \frac{|\nabla\eta|^2}{\eta^{\frac{1}{2}}}\leq C \eta^{\frac{3p+2}{2p+2}}.   \label{eqn:OH12_18}
		\end{align}
		Combining (\ref{eqn:OH12_19}), (\ref{eqn:OH12_17}) and (\ref{eqn:OH12_18}), we obtain 
		\begin{align*}
			\frac{d}{dt}\int_M \eta^\frac{3}{2}R_-^p dv 
			\leq C\int_{M_{-}} \eta^{\frac{3p+2}{2p+2}} R_{-}^{p}dv -\left(\frac{2p}{m}-1\right) \int_{M_{-}} \eta^\frac{3}{2}R_-^{p+1}dv. 
		\end{align*}
	         H\"older's inequality yields that
		\begin{align*}
		  \int_{M_{-}} \eta^{\frac{3p+2}{2p+2}} R_{-}^{p}dv=\int_{M_{-}}  \left\{ \eta^{\frac{3p}{2p+2}} R_{-}^{p} \right\} \cdot \eta^{\frac{1}{p+1}}dv
		  \leq \left(\int_{M_{-}} \eta^\frac{3}{2}R_-^{p+1}dv \right)^{\frac{p}{p+1}}\bigg(\int_{M_{-}} \eta dv \bigg)^{\frac{1}{p+1}}. 
		\end{align*}
		Consequently, we have
		\begin{align}
			\frac{d}{dt}\int_M \eta^\frac{3}{2}R_-^p dv 
			\leq C\left(\int_{M_{-}} \eta^\frac{3}{2}R_-^{p+1}dv \right)^{\frac{p}{p+1}}\bigg(\int_{M_{-}} \eta dv \bigg)^{\frac{1}{p+1}} -\left(\frac{2p}{m}-1\right) \int_M \eta^\frac{3}{2}R_-^{p+1}dv. 
			\label{eqn:OH12_20}
		\end{align}
		Note that $\frac{2p}{m}-1>0$, we can apply Young's inequality to obtain
		\begin{align}
		  C \bigg(\int_M \eta^\frac{3}{2}R_-^{p+1}dv \bigg)^{\frac{p}{p+1}}\bigg(\int_M \eta dv \bigg)^{\frac{1}{p+1}} 
		  \leq  \left(\frac{2p}{m}-1\right) \int_M \eta^\frac{3}{2}R_-^{p+1}dv + C \int_M \eta dv. 
		  \label{eqn:OH12_21}
		\end{align}
		It follows from the combination of (\ref{eqn:OH12_20}) and (\ref{eqn:OH12_21}) that
		\begin{align*}
		   \frac{d}{dt}\int_M \eta^\frac{3}{2}R_-^p dv \leq C \int_M \eta dv, 
		\end{align*}
		which directly implies (\ref{eqn:OH12_14}). \\
		
		\noindent
		\textit{Step 4.  For sufficiently large $L=L(p,m)$, the values of $\int_M \eta dv$ and $\int_M \eta^\frac{3}{2}R_-^p dv$ are dominated by the solutions of the ODE
		 \begin{align}
		 \begin{cases}
		 &\frac{d}{dt}f_1(t)  =  L(f_1(t)+f_2(t)),\\
		 &\frac{d}{dt}f_2(t)  =  L(f_1(t)+f_2(t)),  
		 \end{cases}
		 \label{eqn:OH12_22}
		 \end{align}
	         with the initial data $f_1(0)=\epsilon$, $f_2(0)=V_0$.  }

		Choosing $L=L(p,m)$ as the larger $C$ in (\ref{eqn:OH12_6}) and (\ref{eqn:OH12_14}), we have 
		\begin{equation}
			\left\{
			\begin{aligned}
				\frac{d}{dt}\int_M \eta^\frac{3}{2}R_-^pdv & \leq & L\left( \int_M \eta^\frac{3}{2}R_-^pdv+\int_M \eta dv\right) , \\
				\frac{d}{dt}\int_M \eta dv & \leq & L\left( \int_M \eta^\frac{3}{2}R_-^p dv+\int_M \eta dv\right) .
			\end{aligned}
			\right.
		\end{equation}
	    Define
	    \begin{align}
	     h_1(t) \coloneqq f_1(t)-\int_M \eta^\frac{3}{2}R_-^p dv, \quad
	     h_2(t) \coloneqq f_2(t)-\int_M \eta dv. 
	     \label{eqn:OH12_28}
	    \end{align}
	    In light of (\ref{eqn:OH12_26a}) and (\ref{eqn:OH12_26b}), it is clear that $h_1(0) \geq 0$ and $h_2(0) \geq 0$.  In particular,  we have
	    \begin{align}
	     h_1(0)+h_2(0) \geq 0.  \label{eqn:OH12_23}
	    \end{align}
	    Direct calculation implies that
	     \begin{align}
		 \begin{cases}
		 &\frac{d}{dt}h_1(t)  \geq  L(h_1(t)+h_2(t)),\\
	      &\frac{d}{dt}h_2(t)  \geq  L(h_1(t)+h_2(t)).
		 \end{cases}
		 \label{eqn:OH12_25}
		 \end{align}
	    Thus
	    \begin{align}
	      \frac{d}{dt} \left\{ h_1(t) +h_2(t) \right\} \geq 2L(h_1(t)+h_2(t)).     \label{eqn:OH12_24}
	    \end{align}
	    It follows from (\ref{eqn:OH12_24}) and (\ref{eqn:OH12_23}) that $h_1(t)+h_2(t) \geq 0$ is preserved. 
	    Consequently, the application of (\ref{eqn:OH12_25}) yields that
	    \begin{align*}
	        h_i(t)\geq h_i(0)\geq 0, \quad \forall \; i=1,2, \; t \in [0,1].
	    \end{align*}
	    In light of (\ref{eqn:OH12_28}), the above inequalities mean that
	    \begin{align}
	       \int_M \eta^\frac{3}{2}R_{-}^p dv  \leq  f_1(t), \quad \int_M \eta dv  \leq  f_2(t).    \label{eqn:OH12_29} 
	    \end{align}

	    \vspace{4mm}
	    \noindent
	    \textit{Step 5.  Solve the ODE and finish the proof of (\ref{4.9}) and (\ref{4.8}).}

	    Solving the ODE (\ref{eqn:OH12_22}), it is easy to see that
	    \begin{align}
	      f_1(t)  =  \epsilon +\frac{(V_0+\epsilon)\left( e^{2Lt}-1\right) }{2}, \quad  f_2(t)  =  V_0+\frac{(V_0+\epsilon)\left( e^{2Lt}-1\right) }{2}.
	      \label{eqn:OH12_27}
	    \end{align}	    	
	   Note that $\eta(x,t)=1$ for all $x \in B_{g(0)}(x_0,r_0-2)$ and $t\in [0,1]$.   Thus it follows from (\ref{eqn:OH12_29}) and (\ref{eqn:OH12_27}) that 
	   \begin{align*}
	    & \int_{B_{g(0)}(x_0, r_0-2)} R_{-}^p dv \leq \int_M \eta^\frac{3}{2}R_{-}^p dv  \leq  f_1(t)=\epsilon +\frac{(V_0+\epsilon)\left( e^{2Lt}-1\right) }{2},\\
	    & vol_{g(t)} \left( B_{g(0)}(x_0, r_0-2) \right) \leq  \int_M \eta dv \leq V_0+\frac{(V_0+\epsilon)\left( e^{2Lt}-1\right) }{2}.
	   \end{align*}
	   Replacing $L$ by $C=C(m,p)$ in the above inequalities, we obtain (\ref{4.9}) and (\ref{4.8}).  The proof of Theorem \ref{t4.2} is complete.
	  	\end{proof}

	\begin{corollary}
	
		Same conditions as in Theorem \ref{t4.2}.  Then we have  
		\begin{align}
			\int_{0}^{1}dt\int_{B_{g(0)}(x_0,r_0-2)} R_{-,g(t)}dv\leq \psi(\epsilon |m,p,V_0).
		\label{eqn:OH12_31}	
		\end{align}
		
        \label{cly:OH12_32}		
	\end{corollary}
		
	\begin{proof}
		It follows from (\ref{eqn:OH12_3}) and (\ref{4.9}) that
		\begin{align}
			&\quad \int_{\epsilon}^{1}dt\int_{B_{g(0)}(x_0,r_0-2)} R_{-,g(t)}dv \notag\\
			&\leq \int_{\epsilon}^{1} vol_{g(t)}\left( {B_{g(0)}(x_0,r_0-2)}\right)\frac{\epsilon}{t}dv dt
			   \leq  C V_0\int_{\epsilon}^{1} \frac{\epsilon}{t}dv dt \notag\\
			 &=C V_0\left(-\epsilon \log \epsilon \right)=\psi(\epsilon |m,p,V_0),  \label{4.25}
		\end{align}
		as $\displaystyle \lim_{x \to 0^{+}} x\log x=0$. 		
		Since $\epsilon<<V_0$, it follows from (\ref{4.9}) and (\ref{4.8}) that 
		\begin{align*}
		  vol_{g(t)}\left( {B_{g(0)}(x_0,r_0-2)}\right) +  \int_{B_{g(0)}(x_0,r_0-2)} R_{-,g(t)}^p dv  \leq CV_0  
		\end{align*}
		for each $t \in [0,1]$.  Consequently, we can apply H\"{o}lder inequality to obtain 
		\begin{align}
		  \int_{B_{g(0)}(x_0,r_0-2)} R_{-,g(t)}dv \leq \left(\int_{B_{g(0)}(x_0,r_0-2)} R_{-,g(t)}^p dv\right)^{\frac{1}{p}} \left( vol_{g(t)}\left( {B_{g(0)}(x_0,r_0-2)}\right)\right)^{\frac{p-1}{p}}
		  \leq CV_0 
		  \label{eqn:OH12_30}
		\end{align}
		for each $t \in [0,1]$.  Thus
		\begin{align}
		  \int_{0}^{\epsilon}dt\int_{B_{g(0)}(x_0,r_0-2)} R_{-,g(t)}dv \leq  CV_0 \epsilon=\psi(\epsilon |m,p,V_0).      \label{4.26}
		\end{align}
		Therefore,  (\ref{eqn:OH12_31}) follows from the combination of (\ref{4.25}) and (\ref{4.26}). 
	\end{proof}

	\begin{theorem}
	Same conditions as in Theorem~\ref{thm:OF28_5}. Then we have  
		\begin{align}
		\int_{0}^{1} \int_{B_{g(0)}(x_0, 1)} |R| dv_{g(t)} dt \leq \psi(\delta|m,p). 
		\label{eqn:OH10_2}
		\end{align}
        \label{thm:OF28_8}		
       \end{theorem}
       
       \begin{proof}
       	 
          Since $|R|=R+2R_{-}$, there holds
          \begin{align}
            \int^1_0\int_{B_{g(0)}(x_0,1)}|R|dvdt\leq \int^1_0\int_{B_{g(0)}(x_0,1)}Rdvdt
          +2\int^1_0\int_{B_{g(0)}(x_0,1)}R_-dvdt.
        \label{eqn:OH10_1}  
        \end{align}
         By Theorem \ref{thm:OF28_5}, we have $|Rm|\leq \frac{\epsilon}{t}$ where $\epsilon=\psi(\delta|m,p)$. Therefore, we can apply Corollary~\ref{cly:OH12_32} and set  $r_0=3$. 
           Then we have
          \begin{equation}\label{4.27}
          \int^1_0\int_{B_{g(0)}(x_0,1)}R_-dvdt\leq \psi(\epsilon|m,p),
          \end{equation}
          \begin{equation}\label{4.28}
          \int^1_{\epsilon}\int_{B_{g(0)}(x_0,1)}|R|dvdt\leq \psi(\epsilon|m,p).
          \end{equation}
          Since $\frac{d}{dt}dv=-Rdv$ along the Ricci flow, it is clear that
          \begin{align}
          \int^{\sqrt{\epsilon}}_0\int_{B_{g(0)}(x_0,1)}Rdvdt=vol_{g(0)}\left( B_{g(0)}(x_0,1)\right) -vol_{g(\sqrt{\epsilon})}\left( B_{g(0)}(x_0,1)\right).
          \label{eqn:OH10_3}
          \end{align}
          We already know from given conditions that
          \begin{align}
            vol_{g(0)}(B_{g(0)}(x_0,1))\leq (1+\psi(\epsilon|m,p))\omega_m.   \label{eqn:OH27_14}
          \end{align}
          By Lemma 8.3 of~\cite{Pere}, we have  
          \begin{align*}
            \left(\frac{d}{dt}-\Delta \right) d_{g(t)}(x_0,x)\geq -\frac{2m}{\sqrt{t}},  \quad \textrm{on}  \quad  M \backslash B_{g(t)}\left(x_0, \sqrt{t} \right).  
          \end{align*}
          Thus we have
           \begin{align*}
                B_{g\left(\sqrt{\epsilon}\right) }\left( x_0,1-4m\epsilon^{\frac{1}{4}}\right) \subset B_{g(0)}(x_0,1).
           \end{align*}
          Applying (\ref{eqn:OH27_15b}), we have
          \begin{align}
          &\quad vol_{g(\sqrt{\epsilon})}\left( B_{g(0)}(x_0,1)\right) \notag\\
          &\geq vol_{g(\sqrt{\epsilon})}\left( B_{g(\sqrt{\epsilon})}\left( x_0,1-4m\epsilon^{\frac{1}{4}}\right)\right)
            \geq (1-\psi(\epsilon|m,p))\omega_m\left( 1-4m\epsilon^{\frac{1}{4}}\right)^m.    \label{eqn:OH10_4}
          \end{align}
          Plugging (\ref{eqn:OH10_4}) and (\ref{eqn:OH27_14}) into (\ref{eqn:OH10_3}), we obtain
          \begin{equation}\label{4.29}
          \int^{\sqrt{\epsilon}}_0\int_{B_{g(0)}(x_0,1)}Rdvdt\leq(1+\psi(\epsilon|m,p))\omega_m-(1-\psi(\epsilon|m,p))\omega_m\leq\psi(\epsilon|m,p).
          \end{equation}
          Therefore, (\ref{eqn:OH10_2}) follows from the combination of (\ref{eqn:OH10_1}), (\ref{4.27}), (\ref{4.28}) and (\ref{4.29}).                
       \end{proof}
	
	We remind the reader that Theorem~\ref{thm:OF21_1} was already proved since it follows directly from Theorem~\ref{t4.2} and Corollary~\ref{cly:OH12_32}. 
	It is also clear that Theorem~\ref{thm:OH25_4} follows from Theorem~\ref{thm:OF28_8} up to a standard covering argument.

	\section{Distance distortion and continuous dependence on the initial data}
	\label{sec:distortion}

In Proposition 5.3 of \cite{WB2}, if the lower bound of Ricci curvature  and the volume ratio is close to a Euclidean ball, one can prove that $\int^{1}_0\int_{B_{g(0)}(x_0,1)}|R|dvdt$ is close to 0.  
Namely, we have the local almost Einstein condition, 
which is used in Section 4 of \cite{TW} to obtain the distance distortion estimate. 
In this paper, we only have the Ricci curvature $L^p$-bound, rather than the point-wise lower bound. 
However,  the almost Einstein condition can still be derived, in light of Theorem~\ref{thm:OF28_8} in Section~\ref{sec:SLP}. 
Then the distance distortion estimate can be deduced, following the route in \cite{WB2}.

\begin{lemma}[{\cite{TW}}, Lemma 4.4]
	\label{l5.1}
		Suppose $\left\lbrace (M^m,g(t)), t\in[0,1]\right\rbrace$ is a Ricci flow solution satisfying (\ref{eqn:OF29_6}). 
		Let $\Omega=B_{g(0)}(x_0,1)$, $\Omega'=B_{g(0)}\left( x_0,\frac{1}{2}\right) $ and $\Omega''=B_{g(0)}\left( x_0,\frac{1}{4}\right) $. Define
		\begin{displaymath}
			A_+ \coloneqq \sup_{B_{g(0)}(x_0,r)\subset\Omega',0<r\leq2}\frac{vol_{g(0)}(B_{g(0)}(x_0,r))}{\omega_mr^m},
		\end{displaymath}
		\begin{displaymath}
			A_- \coloneqq \inf_{B_{g(\delta_0)}(x_0,r)\subset\Omega',0<r\leq2}\frac{vol_{g(\delta_0)}(B_{g(\delta_0)}(x_0,r))}{\omega_mr^m}.
		\end{displaymath}
		If $x_1,x_2\in\Omega''$, $l=d_{g(0)}(x_1,x_2)\leq \frac{1}{8}$, then we have
		\[l-CE^{\frac{1}{2m+3}}\leq d_{g(\delta_0)}(x_1,x_2)
		\leq l+ClA_+\left( \left| \frac{A_+}{A_-}-1\right| ^{\frac{1}{m}}+l^{\frac{-1}{m}}E^{\frac{1}{2m(m+3)}}\right), \]
		whenever $E=\int^{2\delta_0}_0\int_{\Omega}|R|dvdt<<l^{2m+3}$.
	\end{lemma}
		
	\begin{proof}
		Checking the proof of Lemma 4.4 in~\cite{TW},  it is not hard to see that only the following conditions 
		\begin{align*}
			|Rm|\leq\frac{\epsilon}{t}, \quad \inf_{\rho\in(0,\epsilon^{-1}\sqrt{t})} \rho^{-m}vol_{g(t)}\big(B_{g(t)}(x,\rho)\big)\geq (1-\epsilon)\omega_m, \quad inj(x,t)\geq \epsilon^{-1}\sqrt{t}
		\end{align*}
		are used in the argument.  Note that the above inequalities are guaranteed by (\ref{eqn:OH27_15}) in Theorem~\ref{thm:OF28_5}. 
		The smallness of $E=\int^{2\delta_0}_0\int_{\Omega}|R|dvdt$ is assured by Theorem~\ref{thm:OF28_8}. 		
		Therefore, the remaining estimate follows verbatim from the one in Lemma 4.4 of~\cite{TW}.  
	\end{proof}
	
	\begin{theorem}[\textbf{Distance distortion of short range}]
		Suppose $p>\frac{m}{2}$ and $\left\lbrace (M^m,g(t)),  0 \leq t \leq r^2 \right\rbrace $ is a Ricci flow solution satisfying
		\begin{subequations}
			\begin{align}[left = \empheqlbrace \,]
				&\int_{B(x_0,\delta^{-1}r)}Rc_-^p dv\leq \frac{1}{2} \omega_m \delta^{4p-m} r^{m-2p},  \label{eqn:OH23_1}  \\
				&\left| B\left(x_0, \delta^{-1}r\right) \right|\geq(1-\delta) \omega_m\delta^{-m} r^{m}.    \label{eqn:OH23_2}
			\end{align}  
		\end{subequations}
		Then we have the short time distance distortion estimate
		\[\frac{|d_{g(\delta^2r^2)}(x,x_0)-d_{g(0)}(x,x_0)|}{\delta r}\leq\psi(\delta|m,p)\] 
		for any $x\in B_{g(0)}\left( x_0,0.1\delta r\right) \setminus B_{g(0)}\left( x_0,0.01\delta r\right)$.
	\label{thm:OH26_2}	
	\end{theorem}
	
	\begin{proof}
		Up to rescaling, we may assume $\delta r=1$.  It suffices to prove
		\begin{align}
		   |d_{g(1)}(x,x_0)-d_{g(0)}(x,x_0)|\leq\psi(\delta|m,p)   \label{eqn:OH09_3}
		\end{align}
		for any $x\in B_{g(0)}(x_0,0.1)\setminus B_{g(0)}(x_0,0.01)$.
		By Theorem \ref{thm:OF28_5}, we have $|Rm|\leq \frac{\epsilon}{t}$ where $\epsilon=\psi(\delta|m,p)$. 
		
		Let $\delta_0=0.1$.  From Theorem \ref{thm:OF28_8}, there holds $\int^{2\delta_0}_0\int_{B_{g(0)}(x_0,1)}|R|dvdt\leq \psi(\delta|m,p)$. In view of Lemma \ref{l5.1}, we have 
		\begin{align}
		   |d_{g(\delta_0)}(x,x_0)-d_{g(0)}(x,x_0)|\leq\psi(\delta|m,p)     \label{eqn:OH09_1}
		\end{align}
		by choosing $(x_1,x_2)=(x_0,x)$ and $l=d_{g(0)}(x,x_0)$.
		On the time interval $[\delta_0,1]$, since the flow is almost stationary, we have 
		\begin{align}
		  |d_{g(1)}(x,x_0)-d_{g(\delta_0)}(x,x_0)|\leq\psi(\delta|m,p).   \label{eqn:OH09_2}
		\end{align}
		Then (\ref{eqn:OH09_3}) follows directly from the combination of (\ref{eqn:OH09_1}) and (\ref{eqn:OH09_2}). 
	\end{proof}
	
	Note that Theorem~\ref{thm:OH26_2} is the counter-part of Proposition 5.3 in~\cite{WB2}. 
	Based on the short range distance distortion estimate in Theorem~\ref{thm:OH26_2} and the pseudo-locality, Theorem~\ref{thm:OF28_5}, 
	one can follow exactly the same argument as in Theorem 5.4 of~\cite{WB2} to obtain the long range distance distortion estimate. 
	Therefore, we state the following theorem without proof. 
		
	\begin{theorem}[\textbf{Distance distortion of long range}]
	
		Same conditions as in Theorem~\ref{thm:OH26_2}.  Then we have the distance distortion estimate
		\begin{align}
		\bigg|\log\frac{d_{g(t)}(x,x_0)}{d_{g(0)}(x,x_0)}\bigg|\leq \psi(\delta|m,p)\bigg(1+\log_+\frac{\sqrt{t}}{d_{g(0)}(x,x_0)}\bigg)    \label{eqn:OH27_5}
		\end{align}
		for any $t \in \left( 0,\delta^2r^2\right) $, $x\in B_{g(0)}(x_0,r)$.  Here $\log_+x=max(0,\log x)$.
	\label{thm:OH26_3}	
	\end{theorem}

         In light of the distance distortion estimates in Theorem~\ref{thm:OH26_2} and Theorem~\ref{thm:OH26_3}, we can prove that the Ricci flow depends continuously on the initial data in the Gromov-Hausdorff topology, when appropriate integral Ricci curvature conditions are satisfied.

	\begin{theorem}[\textbf{Continuous dependence in the $C^\infty$-Cheeger-Gromov topology}]
	
		Suppose that $\mathcal{N}=\left\lbrace \left( N^m, h(t)\right), 0\leq t\leq T \right\rbrace$ is a Ricci flow solution on the closed manifold $N$, initiated from $h(0)=h$, and $p>\frac{m}{2}$. 
		For each $\epsilon$ small, there exists an $\eta=\eta(\mathcal{N}, p,\epsilon)$ with the following properties.
		
		Suppose $\left( M^m, g\right)$ is a Riemannian manifold satisfying 
		\begin{align}
			\kappa(p,\epsilon) <\epsilon^{2-\frac{m}{p}},\quad d_{GH}\left\lbrace (M,g),(N,h)\right\rbrace<\eta.   
		\label{eqn:OH26_4}
		\end{align}
		Then the Ricci flow initiated from $(M,g)$ exists on $\left[0,T \right]$. Furthermore, there exists a family of diffeomorphisms $\left\lbrace \Phi_t:N\longrightarrow M, \epsilon\leq t\leq T\right\rbrace $ such that 
		\begin{equation}
			\sup\limits_{t\in\left[ \epsilon,T\right] }\left\| \Phi_t^*g(t)-h(t)\right\|_{C^{\left[ \epsilon^{-1}\right]}(N,h(t))}<\epsilon. 
		\label{eqn:OH26_5}	
		\end{equation}
	\label{thm:OH26_0}	
	\end{theorem}
	
	\begin{proof}[Sketch of the proof:]
	Since the proof is very similar to that in section 6 of~\cite{WB2}, we shall only sketch the proof and highlight the key points. 
		
	In~\cite{WB2}, the second named author proved estimate (\ref{eqn:OH26_5}) when initial Ricci curvature  has a uniform point-wise lower bound. 
	The key idea there is to construct diffeomorphisms between locally almost flat manifolds with rough Gromov-Hausdorff approximations(cf. Lemma 6.2 and 6.3 of~\cite{WB2}). 
	In the current situation,  the point-wise Ricci lower bound condition 
	$Rc_M>-(m-1)A$ is replaced by the weaker integral Ricci curvature condition $\kappa(p,\epsilon) <\epsilon^{2-\frac{m}{p}}$. 
	The almost flatness and rough Gromov-Hausdorff approximation conditions are realized by the curvature-injectivity-radius estimate in Theorem~\ref{thm:OF28_5}.

	By the smoothness and compactness of $(N, h)$, for each small $\delta$, there exists a small $r_1(\delta,N,h)$ satisfying 
	\begin{align*}
	  \left(1-\frac{\delta}{2} \right)\omega_mr^m\leq \left| B(x,r)\right| \leq \left(1+\frac{\delta}{2} \right)\omega_m r^m, \quad \forall \; x \in N, \; 0<r<r_1.
	\end{align*}
	Choosing $\eta=\eta(h, \epsilon)$ sufficiently small. 
	Since (\ref{eqn:OH26_4}) holds, the volume continuity (cf. Lemma~\ref{l3.7}) and volume comparison (cf. Lemma 2.3 of~\cite{PW1}) guarantees that 
	\begin{align*}
	  (1-\delta)\omega_mr^m\leq \left| B(y,r)\right| \leq (1+\delta)\omega_m r^m, \quad \forall y \in M, \; 0<r<\frac{r_1}{8}. 
	\end{align*}
	Define $r_2 \coloneqq \min\{\epsilon,\frac{r_1}{8}\}$. Then we have 
	\begin{align*}
	\kappa(p,r)\leq 2\left( \frac{r}{\epsilon}\right)^{2-\frac{m}{p}}  \kappa(p,\epsilon)\leq 2r^{2-\frac{m}{p}},  \quad \forall \; 0<r\leq r_2. 
	\end{align*}
	Thus we can choose an $r_3(\delta,\epsilon)$ which satisfies $\kappa(p,r_3)\leq \delta^2$.
	By scaling, on the manifold $(M,\frac{g}{\delta^2 r_3^2})$, (\ref{eqn:OF29_6a}) and (\ref{eqn:OF29_6b}) holds.	  
	Thus we can apply the pseudo-locality property in Theorem~\ref{thm:OF28_5} and obtain the distance distortion estimate in Theorem~\ref{thm:OH26_3}.
	Note that this estimate is the key point (cf. Lemma 6.2 in \cite{WB2}) for the construction of diffeomorphism $\Phi: N \to M$.
	Furthermore,  $\Phi^*(g(\xi))$ is very close to $h$ in $C^0$-topology, for a very small time $\xi=\xi(\mathcal{N}, \epsilon)$.  
	Then we apply the Ricci-Deturck flow technique to obtain (\ref{eqn:OH26_5}), following exactly the same steps as that in Theorem 6.1 of~\cite{WB2}. 
\end{proof}
	
	Theorem \ref{thm:OH26_0}  can be generalized to a version under normalized Ricci flow. 
	The following lemma of volume estimate is needed to achieve such generalization. 
		
	\begin{lemma}\label{l6.1}
		Same conditions as in Theorem \ref{thm:OH26_0}. Suppose $\epsilon$ is a small number satisfying $k(p,\epsilon)\leq \epsilon^{2-\frac{m}{p}}$ and $\left| Rm\right|_{h(t)}\leq \epsilon^{-1}$. 
		Then for each small positive number $\xi<\frac{\epsilon}{100}$, there exists an $\eta=\eta(\xi,\epsilon,h)$ with the following property. 
		
		If $d_{GH}\left\lbrace (M,g), (N,h)\right\rbrace<\eta$, then
		\begin{equation}\label{eqn:OH27_10}
			\left| \log\frac{\left| M\right|_{dv_{g(t)}} }{\left| M\right|_{dv_{g(0)}}}\right| <C(h,m)\epsilon^{-1}\xi,\quad \forall t\in\left( 0,\xi\right). 
		\end{equation}
	\end{lemma}
	
	\begin{proof}
		By volume continuity (cf. (\ref{eqn:OH27_4}) in Lemma~\ref{l3.7}), we have 
		\begin{align*}
		  \left|\left| M\right|_{g(0)}-\left| N\right|_{g(0)} \right| \leq \psi(\eta |\epsilon,h), 
		\end{align*}
		which is equivalent to 
		\begin{align}
			\left| \log\frac{\left| M\right|_{dv_{g(0)}} }{\left| N\right|_{dv_{h(0)}}}\right| \leq  \psi(\eta |\epsilon,h). 
		\label{eqn:OH27_6}	
		\end{align}
		
		By Vitali covering method, there exists a constant $C(m)$ such that 
		$$\left| \left| Rc_-\right| \right|_{p,M} \leq C(m)\epsilon^{-1}\left| M\right|_{g(0)} \leq C(h,m)\epsilon^{-1}.$$ 
		By (\ref{eqn:OH26_4}) and Theorem~\ref{t4.2}, we know for $\xi$ small enough, there holds
		\begin{align*}
			\left\| R_-(t)\right\|_{p,M} <C(h,m)\epsilon^{-1},\quad \forall t\in\left( 0,\xi\right).
		\end{align*}
		Then we calculate
		\begin{align*}
			\frac{d}{dt}\left| M\right|_{dv_{g(t)}}=\int_{M}-Rdv_{g(t)}\leq\int_{M}R_-dv_{g(t)}\leq\left\|R_- \right\|_{p,M} \left| M\right|^{\frac{p-1}{p}}_{dv_{g(t)}} \leq\frac{C(h,m)}{\epsilon}\left| M\right|^{\frac{p-1}{p}}_{dv_{g(t)}},  
		\end{align*}
		which implies
		\begin{align}
			\left( \left| M\right|^{\frac{1}{p}}_{dv_{g(\xi)}} -\frac{C(h,m)(\xi-t)}{\epsilon p}\right)^p\leq\left| M\right|_{dv_{g(t)}}\leq \left( \left| M\right|^{\frac{1}{p}}_{dv_{g(0)}} +\frac{C(h,m)t}{\epsilon p}\right)^p.
		\label{6.5}	
		\end{align}
		As for the Ricci flow $\left( N, h(t)\right)$, it is clear that
		\begin{align*}
		\left| Rm\right|_{h(t)}\leq2\sup\limits_{z\in N}\left| Rm\right|_{h(0)}(z)\leq2\epsilon^{-1},\quad\forall x\in N,\;t\in\left( 0,\xi\right].
		\end{align*}
		Thus direct estimate of volume element yields that
		\begin{align}
			\frac{\left| N\right|_{dv_{h(\xi)}} }{\left| N\right|_{dv_{h(0)}} }\geq e^{-2m(m-1)\epsilon^{-1}\xi}.
		\label{eqn:OH27_8}	
		\end{align}
		On the other hand,  the smooth closeness in Theorem \ref{thm:OH26_0} imply
		\begin{align}
			 \left| \log\frac{\left| M\right|_{dv_{g(\xi)}} }{\left| N\right|_{dv_{h(\xi)}}}\right|\leq \psi(\eta|\xi,\epsilon,m,p). 
		\label{eqn:OH27_7}	 
		\end{align}
				
		For each $t \in (0, \xi)$, it follows from the combination of (\ref{6.5}) and (\ref{eqn:OH27_7}) that
		\begin{align*}
			\left| M\right|_{dv_{g(t)}}^{\frac{1}{p}} &\geq  \left| M\right|^{\frac{1}{p}}_{dv_{g(\xi)}} -\frac{C(h,m)(\xi-t)}{\epsilon p} 
			\geq  \left\{ e^{-\psi} \left| N\right|_{dv_{h(\xi)}}\right\}^{\frac{1}{p}} -\frac{C(h,m)\xi}{\epsilon p}. 
	        \end{align*}
	        Plugging (\ref{eqn:OH27_8}) into the above inequality and applying (\ref{eqn:OH27_6}), we obtain
	       \begin{align*}	
			\quad \left| M\right|_{dv_{g(t)}}^{\frac{1}{p}} 	&\geq  \left\{e^{-\psi-2m(m-1)\epsilon^{-1}\xi} \left| N\right|_{dv_{h(0)}}\right\}^{\frac{1}{p}} -\frac{C(h,m)\xi}{\epsilon p}\\
			   &\geq \left\{ e^{-\psi-2m(m-1)\epsilon^{-1}\xi} \left| M\right|_{dv_{g(0)}}\right\}^{\frac{1}{p}} -\frac{C(h,m)\xi}{\epsilon p}, 
	       \end{align*}	
	       which implies (\ref{eqn:OH27_10}) immediately. 
	\end{proof}
	
	\begin{theorem}[\textbf{Continuous dependence in the case of normalized Ricci flow}]
		Suppose that $\mathcal{N}=\left\lbrace \left( N^m, \tilde{h}(\tilde{t})\right), 0\leq \tilde{t}\leq \tilde{T} \right\rbrace$ is a normalized Ricci flow solution on the closed m-manifold $N$, and $p>\frac{m}{2}$. 
		For each $\epsilon$ small, there exists an $\eta=\eta(\mathcal{N}, p, \epsilon)$ with the following properties.
		
		If $\left( M^m, g\right)$ is a Riemannian manifold satisfying (\ref{eqn:OH26_4}), then the Ricci flow initiated from $(M,g)$ exists on $\left[0,\tilde{T} \right]$. 
		Furthermore, there exists a family of diffeomorphisms $\left\lbrace \tilde{\Phi}_{\tilde{t}}:N\longrightarrow M, \epsilon\leq \tilde{t}\leq \tilde{T}\right\rbrace $ such that 
		\begin{align}
			\sup\limits_{\tilde{t}\in\left[ \epsilon,\tilde{T}\right] }\left\| \tilde{\Phi}_{\tilde{t}}^*g(t)-\tilde{h}(\tilde{t})\right\|_{C^{\left[ \epsilon^{-1}\right]}(N,\tilde{h}(\tilde{t}))}<\epsilon. 
			\label{eqn:OH27_11}
		\end{align}
	\label{thm:OH26_1}	
	\end{theorem}
	\begin{proof}
		Let $\left\lbrace(N,h(t)),t\in[0,T]\right\rbrace$ be the unnormalized Ricci flow corresponding to the normalized Ricci flow $
		\mathcal{N}$. In other words, define
		\begin{align}
			\lambda_N(t) \coloneqq \left(\frac{|N|_{dv_{h(t)}}}{|N|_{dv_{h(0)}}}\right)^{\frac{-2}{m}},\quad
			\tilde{t} \coloneqq \int_0^t\lambda_N(s)ds,\quad 
			\tilde{T} \coloneqq \int_0^T\lambda_N(s)ds,\quad
			\tilde{h}(\tilde{t}) \coloneqq \lambda_N(t)h(t).
		\label{eqn:OH27_12}	
		\end{align}
		Then $\left( N^m, \tilde{h}(\tilde{t})\right)$ satisfies the volume normalized Ricci flow equation and $|N|_{dv_{\tilde{h}(\tilde{t})}}=|N|_{dv_{h(0)}}$. Similarly, we can define $\tilde{g}$ and $\lambda_M$.
		
		Since $\left\lbrace(N,h(t)),t\in[0,T]\right\rbrace$ is a smooth compact space-time, we can extend the Ricci flow to  $\left\lbrace(N,h(t)),t\in[0,T+\epsilon']\right\rbrace$.
		By Theorem~\ref{thm:OH26_0} and Lemma~\ref{l6.1}, for each small $\xi$, the Ricci flow initiated from $(M,g)$ shall exist on $[0,T+\epsilon']$ and satisfy
		\begin{align}
			\sup\limits_{t\in[0,T]}\left|\lambda_M(t)-\lambda_N(t)\right|<\xi.
		\label{eqn:OH27_13}	
		\end{align}
		If $\eta \rightarrow 0$, then $\lambda_M(t)\rightarrow\lambda_N(t)$ uniformly on $[0,T+\epsilon']$.  Consequently, we have
		\begin{align*}
		  \int_0^{T+\epsilon'}\lambda_M(s)ds\rightarrow \int_0^{T+\epsilon'}\lambda_N(s)ds > \int_0^{T} \lambda_N(s)ds=\tilde{T}. 
		\end{align*}
		Thus the  normalized Ricci flow initiated from $(M,g)$ exists on $[0, \tilde{T}]$. 
		Since $\tilde{g}(\tilde{t})=\lambda_{M}(t) g(t)$ and $\lambda_{M}(t)$ is uniformly bounded by (\ref{eqn:OH27_13}), up to adjusting constant slightly,  it is clear that (\ref{eqn:OH27_11}) follows immediately from (\ref{eqn:OH26_5}). 
	\end{proof}
	
        Note that Theorem~\ref{thm:OF21_6} is nothing but the combination of Theorem~\ref{thm:OH26_0} and Theorem~\ref{thm:OH26_1}. 
        Therefore, we have already finished the proof of Theorem~\ref{thm:OF21_6}. 
	
	\section{Proof of the main theorem}
	\label{sec:stability}
	
	In this section, we study the global behavior of the normalized Ricci flow near a given immortal solution initiated from $(N,h)$. 
	In particular, we shall show a stability theorem nearby a weakly stable Einstein manifold under appropriate $L^p$-Ricci curvature conditions.
	Then we apply this stability theorem to prove Theorem~\ref{thm:RG26_3}. 
	
	First we introduce the definition of stability of Ricci flow.
	\begin{definition}\label{d7.1}
		A closed Einstein metric $\left(N,h_E\right)$ is called weakly stable if there exists an $\epsilon=\epsilon\left(N,h_E\right)$ with the following properties.
		
		For any smooth Riemannian metric $h$ satisfying
		\begin{equation}
			|N|_{dv_h}=|N|_{dv_{h_E}},\quad ||h-h_E||_{C^{[\epsilon^{-1}]}\left(h_E\right)}<\epsilon,
		\end{equation}
		the normalized Ricci flow initiated from $(N,h)$ exists immortally and converges (in smooth topology) to an Einstein metric $\left(N,h_E'\right)$.
		
		Furthermore, $\left(N,h_E\right)$ is called strictly stable if each $h_E'$ is isometric to $h_E$ for some sufficiently small $\epsilon$.
	\end{definition}
	
	\begin{theorem}[\textbf{Stability nearby a stable Ricci flow}]\label{t7.2}
		Suppose $\left\lbrace(N^{m},h(t)),t\in[0,\infty)\right\rbrace$ is an immortal solution of normalized Ricci flow with the initial metric $h(0)=h$, and $p>\frac{m}{2}$. 
		The immortal solution $(N,h(t))$ converges to a weakly stable Einstein manifold $\left(N,h_E\right)$. Then for any $\epsilon$ small, there exists an $\eta=\eta\left(N,h,p,\epsilon \right)$ with the following properties.
		
		Suppose a Riemannian manifold $(M,g)$ satisfying
		\begin{equation}\label{7.2}
			\kappa(p,\epsilon) <\epsilon^{2-\frac{m}{p}},\quad d_{GH}\left\lbrace (M,g),(N,h)\right\rbrace<\eta.  
		\end{equation}
		Then the normalized Ricci flow solution initiated from $(M,g)$ exists immortally and converges to an Einstein manifold $\left(N,h_E'\right)$. Moreover, if $h_E$ is strictly stable, then $h_E'$ is homothetic to $h_E$.
	\end{theorem}
	
	\begin{proof}
		Because the immortal solution $(N,h(t))$ converges to an Einstein manifold $\left(N,h_E\right)$, we can choose $T=T(\epsilon)$ such that
		\begin{equation}\label{7.3}
			||h(T)-h_E||_{C^{[\epsilon^{-1}]}\left(h_E\right)}<\epsilon.
		\end{equation}
		
		On the other hand, by Theorem \ref{thm:OH26_1}, we can find $\eta=\eta\left(N,h,p,\epsilon \right)$ such that if $(M,g)$ satisfies 
		\begin{equation*}
			\kappa(p,\epsilon) <\epsilon^{2-\frac{m}{p}},\quad d_{GH}\left\lbrace (M,g),(N,h)\right\rbrace<\eta, 
		\end{equation*}
		then the normalized Ricci flow initiated from $(M,g)$ exists on $[0,T]$ and
		\begin{equation}\label{7.4}
			||h(T)-\Phi^*g(T)||_{C^{[\epsilon^{-1}]}\left(h_E\right)}<\epsilon
		\end{equation}
		for a diffeomorphism $\Phi:N\rightarrow M$. 
		
		Combining (\ref{7.3}) and (\ref{7.4}), we have
		\begin{equation}\label{7.55}
			||h_E-\Phi^*g(T)||_{C^{[\epsilon^{-1}]}\left(h_E\right)}<2\epsilon.
		\end{equation}
		Scaling the metric $(N,\Phi^*g(T))$ slightly, we have
		\begin{equation*}
			|N|_{dv_{\left(C\Phi^*g(T)\right)}}=|N|_{dv_{h_E}},\quad ||h_E-C\Phi^*g(T)||_{C^{[\epsilon^{-1}]}\left(h_E\right)}<3\epsilon
		\end{equation*}
		for an $\epsilon=\epsilon(N,h_E)$ which is sufficiently small. Because $(N,h_E)$ is weakly stable (cf. Definition \ref{d7.1}), 
		the normalized Ricci flow initiated from $(N,\Phi^*g(T))$ exists immortally and  converges to an Einstein manifold $\left(N,h_E'\right)$. If $(N,h_E)$ is strictly stable, then we have $Ch_E'=h_E$ where $C$ is the small scaling number. 
		Concatenating the flows, we obtain a normalized Ricci flow which initiates from $(M,g)$ and converges to $(N,h_E')$.
	\end{proof}
	
	In Theorem \ref{t7.2}, if $(N,h)$ itself is a weakly stable Einstein manifold, then we have the following corollary.
	
	\begin{corollary}[\textbf{Stability nearby a stable Einstein manifold}]
	\label{c7.3}
	Suppose $(N^{m},h)$ is a weakly stable Einstein manifold, $p>\frac{m}{2}$ and $\xi<\epsilon_0(m,p)$ is sufficiently small.
	Then there exists $\delta=\delta(h,p,\xi)$ with the following properties.  
	
	If $(M^{m}, g)$ is a Riemannian manifold satisfying
	\begin{align*}
	 \kappa(p,\xi) <\xi^{2-\frac{m}{p}}, \quad \textrm{and}  \quad
	 d_{GH} \{(M,g), (N,h)\}<\delta, 
	\end{align*}
	then $(M, g)$ can be smoothly deformed to an Einstein metric $(N,h_E)$ by normalized Ricci flow. 
	Moreover, if $(N,h)$ is strictly stable, then the limit metric is homothetic to $(N,h)$.
	\end{corollary}
	
	Now we are ready to finish the proof of main theorem. 
		
	\begin{proof}[Proof of Theorem~\ref{thm:RG26_3}:]
	By the diameter estimate of Aubry~\cite{Aubry}, we know from (\ref{eqn:RG26_1}) that the diameter of $(M, g)$ is bounded by $2\pi$.
	Then we can apply the volume comparison of Petersen-Wei (cf.  Theorem 1.1 of~\cite{PW1}) to obtain
	\begin{align*}
	  \inf_{x \in M}  |B(x, 1)| \geq c_0(m,p), 
	\end{align*}
	which in turn means that the average $L^{p}$-norm of $\{Rc-(m-1)g\}_{-}$ is small.  Namely, (\ref{eqn:OH25_15}) holds. 
	Consequently, the work of Petersen and Sprouse \cite{PeSp} applies and we have 
	\begin{align*}
	   d_{GH}\{(M,g),(S^m,g_{round})\}\leq \psi(\epsilon|m,p). 
	\end{align*}
	Note that the space form $(S^m,g_{round})$ is strictly stable by the work of Huisken~\cite{Huisken85}.
	Therefore, the convergence of the normalized Ricci flow initiated from $(M^{m}, g)$  follows from Corollary \ref{c7.3}.
	The exponential convergence follows from the argument of Hamilton~\cite{Hamilton82}. 
	\end{proof}

\bibliographystyle{amsalpha}

Yuanqing Ma, Institute of Mathematics, The Academy of Mathematics and Systems of Sciences, Chinese Academy of Sciences, Beijing, 100190, China. 

Email: mayuanqing16@mails.ucas.ac.cn
\\

Bing Wang,  Institute of Geometry and Physics, and Key Laboratory of Wu Wen-Tsun Mathematics,
School of Mathematical Sciences, University of Science and Technology of China, No. 96 	Jinzhai Road, Hefei, Anhui Province, 230026, China. 

Email:
topspin@ustc.edu.cn

\end{document}